\documentclass[a4paper,twoside,12pt]{article}
\usepackage[english]{babel}

\usepackage{amsmath,amsfonts,amssymb,amsthm}
\newtheorem{teo}{Theorem}
\newtheorem{lem}{Lemma}
\newtheorem{deff}{Definition}
\newtheorem{remark}{Remark}

\newtheorem{corol}{Corollary}

 \title{{\bf   Abstract Fractional Calculus for  m-accretive Operators  }}

\author{Maksim \,V.~Kukushkin   \\ \\
 \small  \textit{Moscow State University of Civil Engineering, 129337,  Moscow, Russia}\\
 \small\textit{Kabardino-Balkarian Scientific Center, RAS, 360051,  Nalchik, Russia}\\
\textit{\small\textit{kukushkinmv@rambler.ru}} }

\date{}
\begin{document}

\maketitle

\begin{abstract}
In this paper we aim  to construct an abstract  model of a   differential operator with a fractional  integro-differential operator  composition  in  final terms, where  modeling is understood as an  interpretation of   concrete differential  operators in terms of the infinitesimal generator of a corresponding semigroup. We study such operators   as  a Kipriyanov operator, Riesz potential,  difference operator.
Along with this, we consider   transforms  of   m-accretive operators as a generalization,    introduce an  operator class $\mathfrak{G_{\alpha}}$ and provide a description of its  spectral properties.

\end{abstract}
\begin{small}\textbf{Keywords:} Fractional power of an m-accretive  operator; infinitesimal  generator of a semigroup;
 strictly accretive operator; asymptotic formula for the eigenvalues;  Schatten-von Neumann  class.   \\\\
{\textbf{MSC} 47B28; 47A10; 47B12; 47B10; 47B25; 20M05; 26A33.}
\end{small}

\section{Introduction}

To write this paper, we were firstly motivated by  the  boundary value problems of the  Sturm-Liouville type   for fractional
differential equations. Many authors devoted their attention to the topic, nevertheless   this kind of problems are relevant for today. First of all, it is connected  with the fact that they model various physical -
chemical processes: filtration of liquid and gas in highly porous fractal   medium; heat exchange processes in medium  with fractal structure and memory; casual walks of a point particle that starts moving from the origin
by self-similar fractal set; oscillator motion under the action of
elastic forces which is  characteristic for  viscoelastic media, etc.
  In particular,  we would like   to  study  the  eigenvalue problem for a   differential operator  with a fractional derivative in  final terms, in this connection such operators as a Kipriyanov fractional differential operator, Riesz potential,  difference operator are involved.

   In the case corresponding to a selfadjoint senior term we can partially solve the problem  having applied the results of the   perturbation theory,   within the framework of which  the following papers are well-known   \cite{firstab_lit:1Katsnelson}, \cite{firstab_lit:1Krein},   \cite{firstab_lit:2Markus},
  \cite{firstab_lit:3Matsaev},\cite{firstab_lit:Markus Matsaev},
 \cite{firstab_lit:Shkalikov A.}. Generally, to apply the last paper results for a concrete operator $L$ we must be able to represent  it  by  a sum   $L=T+A,$ where the senior term
    $T$ must be either a selfadjoint or normal operator. In other cases we can use methods of the paper  \cite{kukushkin2019},  which are relevant if we deal with non-selfadjoint operators and allow us  to study spectral properties of   operators  whether we have the mentioned above  representation or not.  We should add that the results of  the paper \cite{firstab_lit:Markus Matsaev}  can be  also applied to study non-selfadjoin operators (see a detailed remark in \cite{firstab_lit:Shkalikov A.}).

In many papers
\cite{firstab_lit:1Aleroev1984}-\cite{firstab_lit:1Aleroev1994},   \cite{firstab_lit:1Nakhushev1977}   the eigenvalue problem  was studied by methods of a theory of functions and it is remarkable that  special properties of the fractional derivative were used in these papers, bellow  we present a brief  review.   The singular number   problem for the resolvent of  a second order differential operator
  with  the  Riemann-Liouville  fractional derivative in    final  terms  was considered in the paper \cite{firstab_lit:1Aleroev1984}. It was proved
that the resolvent   belongs to the Hilbert-Schmidt class.  The problem of   completeness
of  the root functions  system was studied  in the paper \cite{firstab_lit:Aleroev1989}, also  similar problems  were considered
in the paper  \cite{firstab_lit:1Aleroev1994}.

However,  we  deal with a more general operator --- a  differential operator    with a fractional  integro-differential operator   composition  in   final terms, which covers the  operator mentioned above. Note that   several types of   compositions of fractional integro-differential operators were studied by such mathematicians as
 Prabhakar T.R. \cite{firstab_lit:1Prabhakar}, Love E.R. \cite{firstab_lit:5Love}, Erdelyi A. \cite{firstab_lit:15Erdelyi}, McBride A. \cite{firstab_lit:9McBride},
  Dimovski I.H., Kiryakova V.S. \cite{firstab_lit:2Dim-Kir}, Nakhushev A.M. \cite{firstab_lit:nakh2003}.

 The central idea of this paper  is to built a   model that gives us a representation  of a  composition of  fractional differential operators   in terms of the semigroup theory.    For instance we can    represent a second order differential operator as some kind of  a  transform of   the infinitesimal generator of a shift semigroup. Continuing this line of reasonings we generalize a   differential operator with a fractional integro-differential composition  in final terms   to some transform of the corresponding  infinitesimal generator   and introduce  a class of  transforms of   m-accretive operators. Further,   we  use   methods obtained in the papers
\cite{firstab_lit(arXiv non-self)kukushkin2018},\cite{kukushkin2019} to  study spectral properties  of non-selfadjoint operators acting  in  a complex  separable Hilbert space, these methods alow us to   obtain an  asymptotic equivalence between   the
real component of the resolvent and the resolvent of the   real component of an operator. Due  to such an approach we  obtain relevant  results since  an asymptotic formula  for   the operator  real component  can be  established in many cases
(see \cite{firstab_lit:2Agranovich2011}, \cite{firstab_lit:Rosenblum}). Thus,   a classification  in accordance with  resolvent  belonging     to  the  Schatten-von Neumann  class is obtained,   a sufficient condition of completeness of the root vectors system is   formulated. As the most significant result  we obtain  an   asymptotic formula for the   eigenvalues.

\section{Preliminaries}

Let    $ C,C_{i} ,\;i\in \mathbb{N}_{0}$ be   real constants. We   assume   that  a  value of $C$ is positive and   can be different in   various formulas  but   values of $C_{i} $ are  certain.   Everywhere further, if the contrary is not stated, we consider   linear    densely defined operators acting on a separable complex  Hilbert space $\mathfrak{H}$. Denote by $ \mathcal{B} (\mathfrak{H})$    the set of linear bounded operators   on    $\mathfrak{H}.$  Denote by
    $\tilde{L}$   the  closure of an  operator $L.$ We establish the following agreement on using  symbols $\tilde{L}^{i}:= (\tilde{L})^{i},$ where $i$ is an arbitrary symbol.  Denote by    $    \mathrm{D}   (L),\,   \mathrm{R}   (L),\,\mathrm{N}(L)$      the  {\it domain of definition}, the {\it range},  and the {\it kernel} or {\it null space}  of an  operator $L$ respectively. The deficiency (codimension) of $\mathrm{R}(L),$ dimension of $\mathrm{N}(L)$ are denoted by $\mathrm{def}\, T,\;\mathrm{nul}\,T$ respectively. Assume that $L$ is a closed   operator acting on $\mathfrak{H},\,\mathrm{N}(L)=0,$  let us define a Hilbert space
$
 \mathfrak{H}_{L}:= \big \{f,g\in \mathrm{D}(L),\,(f,g)_{ \mathfrak{H}_{L}}=(Lf,Lg)_{\mathfrak{H} } \big\}.
$
Consider a pair of complex Hilbert spaces $\mathfrak{H},\mathfrak{H}_{+},$ the notation
$
\mathfrak{H}_{+}\subset\subset\mathfrak{ H}
$
   means that $\mathfrak{H}_{+}$ is dense in $\mathfrak{H}$ as a set of    elements and we have a bounded embedding provided by the inequality
$$
\|f\|_{\mathfrak{H}}\leq C_{0}\|f\|_{\mathfrak{H}_{+}},\,C_{0}>0,\;f\in \mathfrak{H}_{+},
$$
moreover   any  bounded  set with respect to the norm $\mathfrak{H}_{+}$ is compact with respect to the norm $\mathfrak{H}.$
  Let $L$ be a closed operator, for any closable operator $S$ such that
$\tilde{S} = L,$ its domain $\mathrm{D} (S)$ will be called a core of $L.$ Denote by $\mathrm{D}_{0}(L)$ a core of a closeable operator $L.$ Let    $\mathrm{P}(L)$ be  the resolvent set of an operator $L$ and
     $ R_{L}(\zeta),\,\zeta\in \mathrm{P}(L),\,[R_{L} :=R_{L}(0)]$ denotes      the resolvent of an  operator $L.$ Denote by $\lambda_{i}(L),\,i\in \mathbb{N} $ the eigenvalues of an operator $L.$
 Suppose $L$ is  a compact operator and  $N:=(L^{\ast}L)^{1/2},\,r(N):={\rm dim}\,  \mathrm{R}  (N);$ then   the eigenvalues of the operator $N$ are called   the {\it singular  numbers} ({\it s-numbers}) of the operator $L$ and are denoted by $s_{i}(L),\,i=1,\,2,...\,,r(N).$ If $r(N)<\infty,$ then we put by definition     $s_{i}=0,\,i=r(N)+1,2,...\,.$
 According  to the terminology of the monograph   \cite{firstab_lit:1Gohberg1965}  the  dimension  of the  root vectors subspace  corresponding  to a certain  eigenvalue $\lambda_{k}$  is called  the {\it algebraic multiplicity} of the eigenvalue $\lambda_{k}.$
Let  $\nu(L)$ denotes   the sum of all algebraic multiplicities of an  operator $L.$ Let  $\mathfrak{S}_{p}(\mathfrak{H}),\, 0< p<\infty $ be       a Schatten-von Neumann    class and      $\mathfrak{S}_{\infty}(\mathfrak{H})$ be the set of compact operators. By definition, put
$$
\mathfrak{S}_{p}(\mathfrak{H}):=\left\{ L: \mathfrak{H}\rightarrow \mathfrak{H},  \sum\limits_{i=1}^{\infty}s^{p}_{i}(L)<\infty,\;0< p<\infty \right\}.
$$
    Suppose  $L$ is  an   operator with a compact resolvent and
$s_{n}(R_{L})\leq   C \,n^{-\mu},\,n\in \mathbb{N},\,0\leq\mu< \infty;$ then
 we
 denote by  $\mu(L) $   order of the     operator $L$ in accordance with  the definition given in the paper  \cite{firstab_lit:Shkalikov A.}.
 Denote by  $ \mathfrak{Re} L  := \left(L+L^{*}\right)/2,\, \mathfrak{Im} L  := \left(L-L^{*}\right)/2 i$
  the  real  and   imaginary components    of an  operator $L$  respectively.
In accordance with  the terminology of the monograph  \cite{firstab_lit:kato1980} the set $\Theta(L):=\{z\in \mathbb{C}: z=(Lf,f)_{\mathfrak{H}},\,f\in  \mathrm{D} (L),\,\|f\|_{\mathfrak{H}}=1\}$ is called the  {\it numerical range}  of an   operator $L.$
  An  operator $L$ is called    {\it sectorial}    if its  numerical range   belongs to a  closed
sector     $\mathfrak{ L}_{\gamma}(\theta):=\{\zeta:\,|\arg(\zeta-\gamma)|\leq\theta<\pi/2\} ,$ where      $\gamma$ is the vertex   and  $ \theta$ is the semi-angle of the sector   $\mathfrak{ L}_{\gamma}(\theta).$
 An operator $L$ is called  {\it bounded from below}   if the following relation  holds  $\mathrm{Re}(Lf,f)_{\mathfrak{H}}\geq \gamma_{L}\|f\|^{2}_{\mathfrak{H}},\,f\in  \mathrm{D} (L),\,\gamma_{L}\in \mathbb{R},$  where $\gamma_{L}$ is called a lower bound of $L.$ An operator $L$ is called  {\it   accretive}   if  $\gamma_{L}=0.$
 An operator $L$ is called  {\it strictly  accretive}   if  $\gamma_{L}>0.$      An  operator $L$ is called    {\it m-accretive}     if the next relation  holds $(A+\zeta)^{-1}\in \mathcal{B}(\mathfrak{H}),\,\|(A+\zeta)^{-1}\| \leq   (\mathrm{Re}\zeta)^{-1},\,\mathrm{Re}\zeta>0. $
An operator $L$ is called    {\it m-sectorial}   if $L$ is   sectorial    and $L+ \beta$ is m-accretive   for some constant $\beta.$   An operator $L$ is called     {\it symmetric}     if one is densely defined and the following  equality  holds $(Lf,g)_{\mathfrak{H}}=(f,Lg)_{\mathfrak{H}},\,f,g\in   \mathrm{D}  (L).$

    Consider a   sesquilinear form   $ t  [\cdot,\cdot]$ (see \cite{firstab_lit:kato1980} )
defined on a linear manifold  of the Hilbert space $\mathfrak{H}.$   Denote by $   t  [\cdot ]$ the  quadratic form corresponding to the sesquilinear form $ t  [\cdot,\cdot].$
Let   $  \mathfrak{h}=( t + t ^{\ast})/2,\, \mathfrak{k}   =( t - t ^{\ast})/2i$
   be a   real  and    imaginary component     of the   form $  t $ respectively, where $ t^{\ast}[u,v]=t \overline{[v,u]},\;\mathrm{D}(t ^{\ast})=\mathrm{D}(t).$ According to these definitions, we have $
 \mathfrak{h}[\cdot]=\mathrm{Re}\,t[\cdot],\,  \mathfrak{k}[\cdot]=\mathrm{Im}\,t[\cdot].$ Denote by $\tilde{t}$ the  closure   of a   form $t.$  The range of a quadratic form
  $ t [f],\,f\in \mathrm{D}(t),\,\|f\|_{\mathfrak{H}}=1$ is called    {\it range} of the sesquilinear form  $t $ and is denoted by $\Theta(t).$
 A  form $t$ is called    {\it sectorial}    if  its    range  belongs to   a sector  having  a vertex $\gamma$  situated at the real axis and a semi-angle $0\leq\theta<\pi/2.$   Suppose   $t$ is a closed sectorial form; then  a linear  manifold  $\mathrm{D}_{0}(t) \subset\mathrm{D} (t)$   is
called    {\it core}  of $t,$ if the restriction   of $t$ to   $\mathrm{D}_{0}(t)$ has the   closure
$t$ (see\cite[p.166]{firstab_lit:kato1980}).   Due to Theorem 2.7 \cite[p.323]{firstab_lit:kato1980}  there exist unique    m-sectorial operators  $T_{t},T_{ \mathfrak{h}} $  associated  with   the  closed sectorial   forms $t,  \mathfrak{h}$   respectively.   The operator  $T_{  \mathfrak{h}} $ is called  a {\it real part} of the operator $T_{t}$ and is denoted by  $Re\, T_{t}.$ Suppose  $L$ is a sectorial densely defined operator and $t[u,v]:=(Lu,v)_{\mathfrak{H}},\,\mathrm{D}(t)=\mathrm{D}(L);$  then
 due to   Theorem 1.27 \cite[p.318]{firstab_lit:kato1980}   the corresponding  form $t$ is   closable, due to
   Theorem 2.7 \cite[p.323]{firstab_lit:kato1980} there exists   a unique m-sectorial operator   $T_{\tilde{t}}$   associated  with  the form $\tilde{t}.$  In accordance with the  definition \cite[p.325]{firstab_lit:kato1980} the    operator $T_{\tilde{t}}$ is called     a {\it Friedrichs extension} of the operator $L.$

Assume that  $T_{t},\,(0\leq t<\infty)$ is a semigroup of bounded linear operators on    $\mathfrak{H},$    by definition put
$$
Af=-\lim\limits_{t\rightarrow+0} \left(\frac{T_{t}-I}{t}\right)f,
$$
where $\mathrm{D}(A)$  is a set of elements for which  the last limit  exists in the sense of the norm  $\mathfrak{H}.$    In accordance with definition \cite[p.1]{Pasy} the operator $-A$ is called the  {\it infinitesimal  generator} of the semigroup $T_{t}.$

  Let $f_{t} :I\rightarrow \mathfrak{H},\,t\in I:=[a,b],\,-\infty< a <b<\infty.$ The following integral is understood in the Riemann  sense as a limit of partial sums
\begin{equation}\label{11.001}
\sum\limits_{i=0}^{n}f_{\xi_{i}}\Delta t_{i}  \stackrel{\mathfrak{H}}{ \longrightarrow}  \int\limits_{I}f_{t}dt,\,\lambda\rightarrow 0,
\end{equation}
where $(a=t_{0}<t_{1}<...<t_{n}=b)$ is an arbitrary splitting of the segment $I,\;\lambda:=\max\limits_{i}(t_{i+1}-t_{i}),\;\xi_{i}$ is an arbitrary point belonging to $[t_{i},t_{i+1}].$
The sufficient condition of the last integral existence is a continuous property (see\cite[p.248]{firstab_lit:Krasnoselskii M.A.}) i.e.
$
f_{t}\stackrel{\mathfrak{H}}{ \longrightarrow}f_{t_{0}},\,t\rightarrow t_{0},\;\forall t_{0}\in I.
$
The improper integral is understood as a limit
\begin{equation}\label{11.031}
 \int\limits_{a}^{b}f_{t}dt\stackrel{\mathfrak{H}}{ \longrightarrow} \int\limits_{a}^{c}f_{t}dt,\,b\rightarrow c,\,c\in [-\infty,\infty].
\end{equation}

Using     notations of the paper     \cite{firstab_lit:kipriyanov1960} we assume that $\Omega$ is a  convex domain of the  $n$ -  dimensional Euclidean space $\mathbb{E}^{n}$, $P$ is a fixed point of the boundary $\partial\Omega,$
$Q(r,\mathbf{e})$ is an arbitrary point of $\Omega;$ we denote by $\mathbf{e}$   a unit vector having a direction from  $P$ to $Q,$ denote by $r=|P-Q|$   the Euclidean distance between the points $P,Q,$ and   use the shorthand notation    $T:=P+\mathbf{e}t,\,t\in \mathbb{R}.$
We   consider the Lebesgue  classes   $L_{p}(\Omega),\;1\leq p<\infty $ of  complex valued functions.  For the function $f\in L_{p}(\Omega),$    we have
\begin{equation}\label{1a}
\int\limits_{\Omega}|f(Q)|^{p}dQ=\int\limits_{\omega}d\chi\int\limits_{0}^{d(\mathbf{e})}|f(Q)|^{p}r^{n-1}dr<\infty,
\end{equation}
where $d\chi$   is an element of   solid angle of
the unit sphere  surface (the unit sphere belongs to $\mathbb{E}^{n}$)  and $\omega$ is a  surface of this sphere,   $d:=d(\mathbf{e})$  is the  length of the  segment of the  ray going from the point $P$ in the direction
$\mathbf{e}$ within the domain $\Omega.$
Without lose of   generality, we consider only those directions of $\mathbf{e}$ for which the inner integral on the right-hand  side of equality \eqref{1a} exists and is finite. It is  the well-known fact that  these are almost all directions.
 We use a shorthand  notation  $P\cdot Q=P^{i}Q_{i}=\sum^{n}_{i=1}P_{i}Q_{i}$ for the inner product of the points $P=(P_{1},P_{2},...,P_{n}),\,Q=(Q_{1},Q_{2},...,Q_{n})$ which     belong to  $\mathbb{E}^{n}.$
     Denote by  $D_{i}f$  a weak partial derivative of the function $f$ with respect to a coordinate variable with index   $1\leq i\leq n,$ in the
     one-dimensional case  we    use a unified form of notations i.e.     $D_{1}f= d f/dx =f'.$
We  assume that all functions have  a zero extension outside  of $\bar{\Omega}.$
Everywhere further,   unless  otherwise  stated,  we   use  notations of the papers   \cite{firstab_lit:1Gohberg1965},  \cite{firstab_lit:kato1980},  \cite{firstab_lit:kipriyanov1960}, \cite{firstab_lit:1kipriyanov1960},
\cite{firstab_lit:samko1987}.

\vspace{0.5cm}
\noindent{\bf 1. Auxiliary propositions }\\

In this paragraph  we present propositions devoted to properties of  accretive operators and related questions.
For a reader convenience, we would like to establish well-known facts  of the operator theory under  a  point of view  that  is necessary    for the following reasonings.
\begin{lem}\label{L1}
Assume that   $A$ is  a  closed densely defined  operator, the following    condition  holds
\begin{equation}\label{5}
\|(A+t)^{-1}\|_{\mathrm{R} \rightarrow \mathfrak{H}}\leq\frac{1}{ t},\,t>0,
\end{equation}
where a notation  $\mathrm{R}:=\mathrm{R}(A+t)$ is used. Then the operators  $A,A^{\ast}$ are m-accretive.
\end{lem}
\begin{proof}
Using \eqref{5} consider
$$
\|f\|^{2}_{\mathfrak{H}}\leq  \frac{1}{t^{2}}  \|(A+t)f \|^{2}_{ \mathfrak{H}};\,
 \|f\|^{2}_{\mathfrak{H}}\leq  \frac{1}{ t^{2}} \left\{ \| A f\|^{2}_{ \mathfrak{H}}+2t \mathrm{Re}(Af,f)_{ \mathfrak{H}}+t^{2}\|   f \|^{2}_{ \mathfrak{H}}\right\} ;
$$
 $$
t^{-1} \| A  f \|^{2}_{ \mathfrak{H}}+2 \mathrm{Re}(A f,f)_{ \mathfrak{H}}\geq0,\,f\in  \mathrm{D} (A).
$$
Let $t$ be tended to infinity, then we obtain
\begin{equation}\label{6}\mathrm{Re}(A f,f)_{ \mathfrak{H}}\geq0 ,\,f\in \mathrm{D}(A).
\end{equation}
It means  that   the  operator $A$ has an accretive property.
Due to \eqref{6}, we have   $\{\lambda \in \mathbb{C}:\,\mathrm{Re}\lambda<0\}\subset \Delta(A),  $ where $\Delta(A)=\mathbb{C}\setminus   \overline{\Theta  (A)}.$  Applying Theorem 3.2 \cite[p.268]{firstab_lit:kato1980}, we obtain that $A-\lambda$ has a closed range and  $\mathrm{nul} (A-\lambda)=0,\,\mathrm{def} (A-\lambda)=\mathrm{const},\,\forall\lambda\in \Delta(A).$
Let $\lambda_{0}\in \Delta(A) ,\;{\rm Re}\lambda_{0} <0.$
Note that in consequence of inequality  \eqref{6}, we have
 \begin{equation}\label{7.1}
  {\rm Re} ( f,(A-\lambda )f  )_{\mathfrak{H}}\geq   - {\rm Re} \lambda   \|f\|^{2}_{\mathfrak{H}},\,f\in \mathrm{D}(A).
 \end{equation}
  Since the  operator $A-\lambda_{0}$ has a closed range, then
\begin{equation*}
 \mathfrak{H}=\mathrm{R} (A-\lambda_{0})\oplus \mathrm{R} (A-\lambda_{0})^{\perp} .
 \end{equation*}
We remark  that the intersection of the sets  $\mathrm{D}(A)$ and $\mathrm{R} (A-\lambda_{0})^{\perp}$ is  zero, because if we assume  the contrary,   then applying inequality  \eqref{7.1},  for arbitrary element
 $f\in \mathrm{D}(A)\cap \mathrm{R}  (A-\lambda_{0})^{\perp}$    we get
 \begin{equation*}
 - {\rm Re} \lambda_{0}  \|f\|^{2}_{\mathfrak{H}} \leq  {\rm Re} ( f,[A-\lambda_{0} ]f  )_{\mathfrak{H}}=0,
 \end{equation*}
hence $f=0.$   It implies that
$$
\left(f,g\right)_{\mathfrak{H}}=0,\;\forall f\in  \mathrm{R}  (A-\lambda_{0})^{\perp},\;\forall g\in \mathrm{D}(A).
$$
Since $  \mathrm{D}(A)$  is a dense set in $\mathfrak{H},$ then $\mathrm{R}  (A-\lambda_{0})^{\perp}=0.$ It implies that  ${\rm def} (A-\lambda_{0}) =0$ and if we take into account  Theorem 3.2 \cite[p.268]{firstab_lit:kato1980}, then
we came to the conclusion that ${\rm def} (A-\lambda )=0,\;\forall\lambda\in \Delta(A),$ the operator $A$ is m-accretive.

Now assume that the operator $A$ is m-accretive.
 Since it is proved that $\mathrm{def}(A+\lambda)=0,\,\lambda> 0,$ then $ \mathrm{nul}(A+\lambda)^{\ast}=0,\,\lambda> 0$ (see (3.1) \cite[p.267]{firstab_lit:kato1980}).   In accordance with the  well-known fact, we have $\left([\lambda  +A]^{-1}\right)^{\ast}=[\left(\lambda  +A\right)^{\ast}]^{-1}.$ Using the obvious relation $\lambda+A^{\ast}=\left(\lambda  +A\right)^{\ast},$   we can deduce  $(\lambda+A^{\ast})^{-1}=[\left(\lambda  +A\right)^{\ast}]^{-1}.$ Also it is obvious that $\left\|( \lambda  +A)^{-1}\right\| =\left\| [( \lambda  +A)^{-1}]^{\ast}\right\|,$ since both operators are bounded. Hence
$$
\| \left(\lambda  +A^{\ast}\right) ^{-1}f  \|_{\mathfrak{H}}=\|[\left(\lambda  +A\right)^{\ast}]^{-1}f  \|_{\mathfrak{H}}=\|\left([\lambda  +A]^{-1}\right)^{\ast}f\|_{\mathfrak{H}} \leq\frac{1}{ \lambda}\|f\|_{\mathfrak{H}},
 f\in \mathrm{R}(\lambda  +A^{\ast}),\;\lambda>0.
$$
This relation can be rewritten in the following form
\begin{equation*}
\|\left(\lambda  +A^{\ast}\right)^{-1}\|_{\mathrm{R}\rightarrow \mathfrak{H}}  \leq\frac{1}{ \lambda},\,\lambda>0.
\end{equation*}
Using the proved above fact, we   conclude that
\begin{equation}\label{10}
\|\left(\lambda  +A^{\ast}\right)^{-1}\| \leq\frac{1}{\mathrm{Re}\lambda},\,\mathrm{Re}\lambda>0.
\end{equation}
The proof is complete.
\end{proof}

In accordance with the definition given  in \cite{firstab_lit:Krasnoselskii M.A.} we can define  a positive and negative  fractional powers of a positive  operator $A$ as follows
\begin{equation}\label{9}
A^{\alpha}:=\frac{\sin\alpha \pi}{\pi}\int\limits_{0}^{\infty}\lambda^{\alpha-1}(\lambda  +A)^{-1} A \,d \lambda;\,\,A^{-\alpha}:=\frac{\sin\alpha \pi}{\pi}\int\limits_{0}^{\infty}\lambda^{-\alpha}(\lambda  +A)^{-1}  \,d \lambda,\,\alpha\in (0,1).
\end{equation}
This definition can be correctly extended  on m-accretive operators, the corresponding reasonings can be found in \cite{firstab_lit:kato1980}. Thus, further we define positive and negative fractional powers of m-accretive operators by formula \eqref{9}.

 \begin{lem}\label{L3.1}
 Assume that $\alpha\in(0,1),$ the operator   $J$ is m-accretive,    $  J^{-1} $ is bounded, then
\begin{equation}\label{15.09}
\|J^{-\alpha}f\|_{\mathfrak{H}}\leq C_{   1-\alpha}\|f\|_{\mathfrak{H}},\,f\in \mathfrak{H},
\end{equation}
where
$
C_{ 1-\alpha }= 2 (1-\alpha)^{-1} \|J^{-1}\|  + \alpha^{-1} .$
\end{lem}
\begin{proof}
 Consider
$$
J^{-\alpha}=\int\limits_{0}^{1}\lambda^{-\alpha }(\lambda+J)^{-1}d\lambda+\int\limits_{1}^{\infty}\lambda^{-\alpha}(\lambda+J)^{-1}d\lambda=I_{1}+I_{2}.
$$
Using definition of the integral \eqref{11.001},\eqref{11.031}  in a Hilbert  space  and the fact $J(\lambda+J)^{-1}f=(\lambda+J)^{-1}Jf,\,f\in \mathrm{D}(J),$ we can easily obtain
$$
\|I_{1}f\|_{\mathfrak{H}}=\left\|\int\limits_{0}^{1}\lambda^{-\alpha}J^{-1}J(\lambda+J)^{-1}fd\lambda\right\|_{\mathfrak{H}}\leq
\|J^{-1}\|  \cdot\left\|\int\limits_{0}^{1}\lambda^{-\alpha}J(\lambda+J)^{-1}fd\lambda\right\|_{\mathfrak{H}}\leq
$$
$$
\leq \|J^{-1}\|_{\mathrm{R} \rightarrow \mathfrak{H}} \cdot \left\{\left\| f \right\|_{\mathfrak{H}}  \int\limits_{0}^{1}\lambda^{-\alpha} fd\lambda  +  \left\|\int\limits_{0}^{1}\lambda^{1-\alpha} (\lambda+J)^{-1}fd\lambda\right\|_{\mathfrak{H}}   \right\}       \leq
 $$
 $$
 \leq2 \|J^{-1}\|\cdot\left\|  f \right\|_{\mathfrak{H}} \int\limits_{0}^{1}\lambda^{-\alpha} d\lambda
,\,f\in \mathrm{D}(J);
$$
$$
\|I_{2}f\|_{\mathfrak{H}}=\left\|\int\limits_{1}^{\infty}\lambda^{-\alpha} (\lambda+J)^{-1}fd\lambda\right\|_{\mathfrak{H}}\leq
\left\|  f \right\|_{\mathfrak{H}}    \int\limits_{1}^{\infty}\lambda^{-\alpha} \|(\lambda+J)^{-1}\| d\lambda  \leq\left\|  f \right\|_{\mathfrak{H}} \int\limits_{1}^{\infty}\lambda^{-\alpha-1}   d\lambda .
$$
Hence
$ J^{-\alpha}
$  is bounded on $\mathrm{D}(J).$  Since $\mathrm{D}(J)$ is dense in $\mathfrak{H},$ then $J^{-\alpha}$ is bounded on $\mathfrak{H}.$ Calculating the right-hand sides of the above estimates, we obtain \eqref{15.09}.
\end{proof}

\section{Main results}
In this section we  explore  a special  operator class for which    a number of  spectral theory theorems can be applied. Further we construct an abstract  model of a  differential operator in terms of m-accretive operators and call it an m-accretive operator transform, we  find  such conditions that    being  imposed guaranty  that the transform   belongs to the class.  As an application of the obtained abstract results  we study a differential  operator   with a fractional  integro-differential operator   composition  in  final terms on a bounded domain of the $n$ -  dimensional Euclidean space as well as on real axis. One of the central points is a relation connecting fractional powers of m-accretive operators and fractional derivative in the most general sense. By virtue of such an approach we express fractional derivatives in terms of   infinitesimal generators, in this regard such operators as    a Kipriyanov operator, Riesz potential,  difference operator are considered.

\vspace{0.5cm}
\noindent{\bf 1. Spectral theorems}\\

 Bellow, we give  a slight generalization of the results presented in  \cite{firstab_lit(arXiv non-self)kukushkin2018}.
\begin{teo}\label{T1a}
Assume that $L$ is a non-sefadjoint operator acting in $\mathfrak{H},$  the following  conditions hold\\

 \noindent ($ \mathrm{H}1 $) There  exists a Hilbert space $\mathfrak{H}_{+}\subset\subset\mathfrak{ H}$ and a linear manifold $\mathfrak{M}$ that is  dense in  $\mathfrak{H}_{+}.$ The operator $L$ is defined on $\mathfrak{M}.$    \\

 \noindent  $( \mathrm{H2} )  \,\left|(Lf,g)_{\mathfrak{H}}\right|\! \leq \! C_{1}\|f\|_{\mathfrak{H}_{+}}\|g\|_{\mathfrak{H}_{+}},\,
      \, \mathrm{Re}(Lf,f)_{\mathfrak{H}}\!\geq\! C_{2}\|f\|^{2}_{\mathfrak{H}_{+}} ,\,f,g\in  \mathfrak{M},\; C_{1},C_{2}>0.
$
\\

\noindent Let $W$ be a restriction of the operator  $L$ on the set $\mathfrak{M}.$  Then   the following  propositions are true.\\

\noindent({\bf A}) {\it We have the following classification
\begin{equation*}
R_{ \tilde{W} }\in  \mathfrak{S}_{p},\,p= \left\{ \begin{aligned}
\!l,\,l>2/\mu,\,\mu\leq1,\\
   1,\,\mu>1    \\
\end{aligned}
 \right.\;,
\end{equation*}
where $\mu$ is    order of $H:=Re \, \tilde{W} .$
Moreover  under  the  assumptions $ \lambda_{n}(R_{H})\geq  C \,n^{-\mu},\,n\in \mathbb{N},$  we have
$$
 R_{  \tilde{W}}\in\mathfrak{S}_{p}\;  \Rightarrow \;\mu p>1,\;1\leq p<\infty.
$$
}\\

  \noindent({\bf B}) {\it The following relation  holds
 \begin{equation}\label{4}
\sum\limits_{i=1}^{n}|\lambda_{i}(R_{ \tilde{W}})|^{p}\leq
  C \sum\limits_{i=1}^{n } \,\lambda^{ p}_{i}(R_{H})  ,\,1\leq p<\infty,\, \;(n=1,2,...,\, \nu(R_{ \tilde{W} })) ,
\end{equation}
moreover   if  $\nu(R_{ \tilde{W} })=\infty$ and      $\mu \neq0,$  then  the following asymptotic formula  holds
$$
|\lambda_{i}(R_{\tilde{W}})|=  o\left(i^{-\mu+\varepsilon}    \right)\!,\,i\rightarrow \infty,\;\forall \varepsilon>0.
$$}

\noindent({\bf C})  {\it Assume that   $\theta< \pi \mu/2\, ,$    where $\theta$ is   the   semi-angle of the     sector $ \mathfrak{L}_{0}(\theta)\supset \Theta (\tilde{W}).$
Then  the system of   root   vectors  of   $R_{ \tilde{W} }$ is complete in $\mathfrak{H}.$}

\end{teo}

\begin{proof}
Note that due to the first  condition  H2,    by virtue of   Theorem 3.4 \cite[p.268]{firstab_lit:kato1980} the  operator  $W$  is closable.
Let us show that $\tilde{W}$ is sectorial. By virtue of condition H2, we get
$$
\mathrm{Re}(\tilde{W}f,f)_{\mathfrak{H}}\geq C_{2}\|f\|^{2}_{\mathfrak{H_{+}}}\geq C_{2}\varepsilon\|f\|^{2}_{\mathfrak{H}_{+}}+\frac{C_{2}(1-\varepsilon)}{C_{0}}\|f\|^{2}_{\mathfrak{H}};
$$
$$
\mathrm{Re}(\tilde{W}f,f)_{\mathfrak{H}} -k|\mathrm{Im}(\tilde{W}f,f)_{\mathfrak{H}}|\geq (C_{2}\varepsilon-k C_{1})\|f\|^{2}_{\mathfrak{H}_{+}}  +\frac{C_{2}(1-\varepsilon)}{C_{0}}\|f\|^{2}_{\mathfrak{H}}=\frac{C_{2}(1-\varepsilon)}{C_{0}}\|f\|^{2}_{\mathfrak{H}},
$$
where
$
 k=  \varepsilon C_{2}/C_{1}.
$
Hence $\Theta(\tilde{W})\subset \mathfrak{L}_{\gamma}(\theta),\,\gamma=C_{2}(1-\varepsilon)/C_{0}.$ Thus, the claim of Lemma  1 \cite{firstab_lit(arXiv non-self)kukushkin2018}  is true regarding   the operator $\tilde{W}.$   Using this fact, we conclude that the claim of   Lemma   2 \cite{firstab_lit(arXiv non-self)kukushkin2018} is true regarding the operator $\tilde{W}$ i.e.  $\tilde{W}$ is m-accretive.

Using the first representation theorem (Theorem 2.1 \cite[p.322]{firstab_lit:kato1980})  we have a one-to-one correspondence between m-sectorial operators and closed sectorial sesquilinear forms i.e. $\tilde{W}=T_{t}$ by symbol,  where $t$ is a sesquilinear form corresponding to the operator $\tilde{W}.$  Hence $H:=Re\, \tilde{W}$ is defined (see \cite[p.337]{firstab_lit:kato1980}). In accordance with Theorem 2.6 \cite[p.323]{firstab_lit:kato1980} the operator $H$ is selfadjoint, strictly accretive.

 A compact embedding   provided by the relation $\mathfrak{h}[f]\geq C_{2} \|f\|_{\mathfrak{H}_{+}} \geq C_{2}/C_{0}\|f\|_{\mathfrak{H}},\,f\in \mathrm{D}(h)$ proves that $R_{H}$ is compact (see proof of  Theorem  4  \cite{firstab_lit(arXiv non-self)kukushkin2018}) and as a result   of  the application of  Theorem 3.3   \cite[p.337]{firstab_lit:kato1980}, we get $R_{\tilde{W}}$ is compact. Thus the  claim of  Theorem  4  \cite{firstab_lit(arXiv non-self)kukushkin2018} remains true regarding   the operators $R_{H},\,R_{\tilde{W}}.$

 In accordance with   Theorem 2.5 \cite[p.323]{firstab_lit:kato1980} , we get $ W ^{\ast}=T_{t^{\ast}}$  (since $W^{\ast}=\tilde{W}^{\ast}$).
Now if we denote $t_{1}:=t^{\ast},$  then  it is easy to calculate $\mathfrak{k} =-\mathfrak{k}_{1}.$   Since $t$ is sectorial, than $|\mathfrak{k}_{1}|\leq \tan \theta\cdot \mathfrak{h}.$ Hence, in accordance with Lemma 3.1 \cite[p.336]{firstab_lit:kato1980}, we get  $\mathfrak{k} [u,v]=(BH^{1/2}u,H^{1/2}v),\,\mathfrak{k}_{1}[u,v]=-(BH^{1/2}u,H^{1/2}v),\,u,v\in \mathrm{D}(H^{1/2}),$ where    $B\in  \mathcal{B}(\mathfrak{H}) $ is a symmetric operator. Let us prove that $B$ is selfadjoint.  Note that in accordance with  Lemma 3.1 \cite[p.336]{firstab_lit:kato1980} $\mathrm{D}(B)=\mathrm{R}(H^{1/2}),$    in accordance with   Theorem 2.1 \cite[p.322]{firstab_lit:kato1980}, we have   $(Hf,f)_{\mathfrak{H}}\geq C_{2}/C_{0}\|f\|_{\mathfrak{H}}^{2},\,f\in \mathrm{D}(H),$ using the reasonings of Theorem  5 \cite{firstab_lit(arXiv non-self)kukushkin2018}, we conclude that   $\mathrm{R}(H^{1/2})=\mathfrak{H}$ i.e. $\mathrm{D}(B)=\mathfrak{H}.$  Hence $B$ is selfadjoint.  Using Lemma 3.2 \cite[p.337]{firstab_lit:kato1980}, we obtain a representation $\tilde{W}=H^{1/2}(I+iB) H^{1/2},\,W^{\ast}=H^{1/2}(I-iB) H^{1/2}.$  Noting the fact   $\mathrm{D}(B)=\mathfrak{H},$  we can easily obtain $ (I\pm iB)^{\ast}= I \mp iB.$ Since $B$ is selfadjoint, then
$\mathrm{Re}([I\pm iB]f,f)_{\mathfrak{H}}=\|f\|^{2}_{\mathfrak{H}}.$ Using this fact and  applying Theorem 3.2 \cite[p.268]{firstab_lit:kato1980}, we conclude that $\mathrm{R}(I\pm iB)$ is a closed set. Since $\mathrm{N}(I\pm iB)=0,$ then $\mathrm{R}(I \mp iB)=\mathfrak{H}$ (see (3.2) \cite[p.267]{firstab_lit:kato1980}). Thus, we obtain   $(I\pm i B)^{-1}\in \mathcal{B}(\mathfrak{H}).$ Taking into account the above facts, we get $R_{\tilde{W}}=H^{-1/2}(I+iB)^{-1} H^{-1/2},\,R_{W^{\ast}}=H^{-1/2}(I-iB)^{-1} H^{-1/2}.$ In accordance with the well-known theorem (see Theorem 5 \cite[p.557]{firstab_lit:Smirnov5}), we have $R^{\ast}_{\tilde{W} }=R_{W^{\ast}}.$ Note that the relations   $(I\pm i B)\in \mathcal{B}(\mathfrak{H}),\,(I\pm i B)^{-1}\in \mathcal{B}(\mathfrak{H}),\,H^{-1/2}\in \mathcal{B}(\mathfrak{H})$  allow  as to obtain the following formula by direct calculations
$$
\mathfrak{Re} R _{\tilde{W} }    =\frac{1}{2}H^{-1/2}(I+ B^{2})^{-1} H^{-1/2}.
$$
  This formula is a crucial point of the matter, we can repeat the rest part of the proof of    Theorem  5 \cite{firstab_lit(arXiv non-self)kukushkin2018}  in terms $H:= Re \, \tilde{W}.$  By virtue of these facts  Theorems  7-9 \cite{firstab_lit(arXiv non-self)kukushkin2018},  can be  reformulated in terms  $H:= Re \, \tilde{W} ,$    since they are based on Lemmas  1,  3, Theorems  4,  5 \cite{firstab_lit(arXiv non-self)kukushkin2018}.\\
\end{proof}

\begin{remark}\label{R1}
Consider  a condition  $\mathfrak{M}\subset \mathrm{D}( W ^{\ast}),$ in this case the operator $\mathcal{H}:=\mathfrak{Re }\,W$ is defined on $\mathfrak{M},$ the fact is that $\tilde{\mathcal{H}}$ is selfadjoint,    bounded  from bellow (see Lemma  3 \cite{firstab_lit(arXiv non-self)kukushkin2018}). Hence a corresponding sesquilinear  form (denote this form by $h$) is symmetric and  bounded from bellow also (see Theorem 2.6 \cite[p.323]{firstab_lit:kato1980}). It can be easily shown  that $h\subset   \mathfrak{h},$  but using this fact    we cannot claim in general that $\tilde{\mathcal{H}}\subset H$ (see \cite[p.330]{firstab_lit:kato1980} ). We just have an inclusion   $\tilde{\mathcal{H}}^{1/2}\subset H^{1/2}$     (see \cite[p.332]{firstab_lit:kato1980}). Note that the fact $\tilde{\mathcal{H}}\subset H$ follows from a condition $ \mathrm{D}_{0}(\mathfrak{h})\subset \mathrm{D}(h) $ (see Corollary 2.4 \cite[p.323]{firstab_lit:kato1980}).
 However, it is proved (see proof of Theorem  4 \cite{firstab_lit(arXiv non-self)kukushkin2018}) that relation H2 guaranties that $\tilde{\mathcal{H}}=H.$ Note that the last relation is very useful in applications, since in most concrete cases we can find a concrete form of the operator $\mathcal{H}.$

\end{remark}

\vspace{0.5cm}

\noindent{\bf 2. Transform}\\

Consider a transform of an m-accretive operator $J$ acting in $\mathfrak{H}$
\begin{equation}\label{12.0.1}
 Z^{\alpha}_{G,F}(J):= J^{\ast}GJ+FJ^{\alpha},\,\alpha\in [0,1),
\end{equation}
where symbols  $G,F$  denote  operators acting in $\mathfrak{H}.$    Further, using a relation $L= Z^{\alpha}_{G,F}(J)$ we mean that there exists an appropriate representation for the operator $L.$
\noindent The following theorem gives us a tool to describe spectral properties  of transform  \eqref{12.0.1},
  as it will be shown   further  it has an important  application in fractional calculus  since  allows   to represent fractional differential  operators as a transform of the infinitesimal  generator of a    semigroup.

\begin{teo}\label{T1}
 Assume that  the operator   $J$ is m-accretive,    $  J^{-1} $ is compact, $G$ is bounded, strictly accretive, with  a lower bound $\gamma_{G}> C_{ \alpha} \|J^{-1}\|\cdot \|F\|,\; \mathrm{D}(G)\supset \mathrm{R}(J),$   $F\in \mathcal{B}(\mathfrak{H}),$ where $C_{\alpha}$  is a   constant   \eqref{15.09}.
 Then   $Z^{\alpha}_{G,F}(J)$ satisfies  conditions  H1 -  H2.

\end{teo}
\begin{proof}
Since  $J$ is m-accretive, then  it  is    closed, densely defined (see \cite[p.279]{firstab_lit:kato1980},  using the fact that $(J+\lambda)^{-1},\,(\lambda>0)$ is a closed operator, we conclude that $J$ is closed also).
Firstly,  we want to check fulfilment of condition $\mathrm{H1}.$  Let us choose  a space $\mathfrak{H}_{J}$ as a space   $\mathfrak{H}_{+}.$        Since $J^{-1}$ is compact, then   we conclude that the following   relation holds
 $\|f\|_{\mathfrak{H}}\leq \|J^{-1}\| \cdot \|Jf\|_{\mathfrak{H}},\,f\in \mathrm{D}(J) $ and    the embedding provided by this   inequality is compact.    Thus condition H1 is satisfied.

Let us prove that  $\mathrm{D}(J^{\ast}GJ)$ is a core of $J.$
Consider a space $\mathfrak{H}_{J}$ and a sesquilinear form
$$
 l_{G}(u,v):=(GJu,Jv)_{\mathfrak{H}} ,\;u,v\in \mathrm{D}(J).
$$
Observe that this form is a bounded functional on $\mathfrak{H}_{J},$ since we have
$
 |(GJu,Jv)_{\mathfrak{H}}|\leq \|G\|\cdot \|Ju\|_{\mathfrak{H}}  \|Jv\|_{\mathfrak{H}}.
$
Hence using the Riesz representation theorem, we have
$$
\forall z\in \mathrm{D}(J),\,   \exists f\in \mathrm{D}(J):\,(GJz,Jv)_{\mathfrak{H}}=(Jf,Jv)_{\mathfrak{H}}.
$$
On the other hand,   due to the properties of the operator $G,$ it is clear that  the conditions of the Lax-Milgram theorem are satisfied  i.e.
$
|(GJu,Jv)_{\mathfrak{H}}|\leq \|G\|\cdot\|Ju\|_{\mathfrak{H}} \|Jv\|_{\mathfrak{H}},\;|(GJu,Ju)_{\mathfrak{H}}|\geq \gamma_{G} \|Ju\|^{2}_{\mathfrak{H}}.
$
 Note that,   in accordance with  Theorem 3.24 \cite[p.275]{firstab_lit:kato1980} the set
$  \mathrm{D} (J^{\ast}J)$ is a core of $J$  i.e.
$$
\forall f\in \mathrm{D}(J),\,\exists \{f_{n}\}_{1}^{\infty}\subset \mathrm{D} (J^{\ast}J):   \,f_{n}\xrightarrow[ J ]{}f.
$$
 Using the Lax-Milgram theorem, in the previously used terms, we get
$$
   \forall f_{n},\, n\in \mathbb{N},\,    \exists z_{n}\in \mathrm{D}(J):\,(GJz_{n},Jv)_{\mathfrak{H}}=(Jf_{n},Jv)_{\mathfrak{H}}.
$$
Combining the above relations, we obtain
$$
(GJ\xi_{n},Jv)_{\mathfrak{H}}=(J\psi_{n},Jv)_{\mathfrak{H}},\,
$$
where $\xi_{n}:=z-z_{n},\,\psi_{n}:=f-f_{n}.$
  Using the strictly accretive property of the operator $G,$ we have
$$
\|J\xi_{n}\|_{\mathfrak{H}}^{2}\gamma_{G} \leq|(GJ\xi_{n},J\xi_{n})_{\mathfrak{H}}|=|(J\psi_{n},J\xi_{n})_{\mathfrak{H}}|\leq\|J\psi_{n}\|_{\mathfrak{H}}\|J\xi_{n}\|_{\mathfrak{H}} .
$$
Taking into account that $J^{-1}$ is bounded, we obtain
$$
K_{1}\| \xi_{n}\|_{\mathfrak{H}}\leq\|J\xi_{n}\|_{\mathfrak{H}} \leq K_{2}\|J\psi_{n}\|_{\mathfrak{H}},\; K_{1},K_{2}>0,
$$
from what follows that
$$
  Jz_{n} \xrightarrow[   ]{\mathfrak{H}}Jz .
$$
On the other hand, we have
$$
(GJz_{n},Jv)_{\mathfrak{H}}=(Jf_{n}, Jv)_{\mathfrak{H}}=(J^{\ast}Jf_{n}, v)_{\mathfrak{H}},\,v\in \mathrm{D}(J).
$$
Hence $\{z_{n}\}_{1}^{\infty}\subset \mathrm{D}(J^{\ast}GJ).$
Taking into account  the above reasonings, we conclude that $\mathrm{D}(J^{\ast}GJ)$ is a core of $J.$ Thus, we have obtained the desired result.

   Note that $\mathrm{D}_{0}(J)$ is dense in $\mathfrak{H},$ since $J$ is densely defined. We have proved above
$$
\mathrm{Re}\left(J^{\ast}GJf,f\right)_{\mathfrak{H}}=\mathrm{Re}\left( GJf,Jf\right)_{\mathfrak{H}}\geq  \gamma_{G}\|f\|^{2}_{\mathfrak{H}_{J}},\,
$$
$$
 \left|\left(J^{\ast}GJf,g\right)_{\mathfrak{H}}\right|=\left|\left( GJf,Jg\right)_{\mathfrak{H}}\right|\leq \|G\|\cdot\|Jf\|_{\mathfrak{H}}\|Jg\|_{\mathfrak{H}},\,f,g\in  \mathrm{D}_{0}(J).
$$
Similarly, we get
\begin{equation}\label{15}
|(FJ^{\alpha}f,g)_{\mathfrak{H}}|\leq \|FJ^{\alpha}f\|_{\mathfrak{H}}\|g\|_{\mathfrak{H}}\leq \|J^{-1}\|\cdot\|F\|\cdot\| J^{\alpha}f\|_{\mathfrak{H}}\|Jg\|_{\mathfrak{H}},\,f,g\in \mathrm{D}_{0}(J) .
\end{equation}
In accordance with \eqref{9}, we have $J^{\alpha-1}J\subset J^{\alpha}.$
 Therefore, using Lemma \ref{L3.1}, we obtain
\begin{equation}\label{15.0.1}
\|J^{\alpha}f\|_{\mathfrak{H}}= \|J^{\alpha-1}Jf\|_{\mathfrak{H}}\leq C_{  \alpha  }\| Jf\|_{\mathfrak{H}},\,f\in \mathrm{D}_{0}(J).
\end{equation}
  Combining this fact with \eqref{15},  we obtain
\begin{equation*}
|(FJ^{\alpha}f,g)_{\mathfrak{H}}| \leq  C_{  \alpha }\|J^{-1}\|\cdot \|F\|\cdot \|  f\|_{\mathfrak{H}_{J}}\|g\|_{\mathfrak{H}_{J}},\,f,g\in \mathrm{D}_{0}(J),
\end{equation*}
(the  case corresponding to $ \alpha=0 $ is trivial, since the operator $J^{-1}$ is bounded).
 It follows that
$$
\mathrm{Re}(FJ^{\alpha}f,f)\geq- C_{ \alpha}\|J^{-1}\|\cdot \|F\|\cdot \|  f\|^{2}_{\mathfrak{H}_{J}},\,f\in \mathrm{D}_{0}(J).
$$
Combining the above facts, we obtain fulfillment of  condition  $ \mathrm{H2}.$

\end{proof}

\begin{deff}
  Define an operator class $\mathfrak{G_{\alpha}} :=\{W:\,W\!= Z^{\alpha}_{G,F}(J) \},$ where $G,F,J$ satisfy  the conditions   of Theorem \ref{T1}.
\end{deff}

\vspace{0.5cm}

\noindent{\bf 3. The  model}\\

In this section we consider various   operators  acting in a complex separable  Hilbert space  for which   Theorem \ref{T1a} can be applied,  the given bellow results also   cover  a case $\alpha=0$ after   minor changes  which are omitted due to simplicity.  In accordance with Remark \ref{R1}, we will stress cases when the relation $\tilde{\mathcal{H}}=H$ can be obtained.  \\

\noindent {\bf   Kipriyanov operator }\\

Here, we study a case $\alpha\in (0,1).$ Assume that  $\Omega\subset \mathbb{E}^{n}$ is  a convex domain, with a sufficient smooth boundary ($ C ^{3}$ class)   of the n-dimensional Euclidian space. For the sake of the simplicity we consider that $\Omega$ is bounded, but  the results  can be extended     to some type of    unbounded domains.
In accordance with the definition given in  the paper  \cite{firstab_lit:1kukushkin2018}, we consider the directional  fractional integrals.  By definition, put
$$
 (\mathfrak{I}^{\alpha}_{0+}f)(Q  ):=\frac{1}{\Gamma(\alpha)} \int\limits^{r}_{0}\frac{f (P+t \mathbf{e} )}{( r-t)^{1-\alpha}}\left(\frac{t}{r}\right)^{n-1}\!\!\!\!dt,\,(\mathfrak{I}^{\alpha}_{d-}f)(Q  ):=\frac{1}{\Gamma(\alpha)} \int\limits_{r}^{d }\frac{f (P+t\mathbf{e})}{(t-r)^{1-\alpha}}\,dt,
$$
$$
\;f\in L_{p}(\Omega),\;1\leq p\leq\infty.
$$
Also,     we   consider auxiliary operators,   the so-called   truncated directional  fractional derivatives    (see \cite{firstab_lit:1kukushkin2018}).  By definition, put
 \begin{equation*}
 ( \mathfrak{D} ^{\alpha}_{0+,\,\varepsilon}f)(Q)=\frac{\alpha}{\Gamma(1-\alpha)}\int\limits_{0}^{r-\varepsilon }\frac{ f (Q)r^{n-1}- f(P+\mathbf{e}t)t^{n-1}}{(  r-t)^{\alpha +1}r^{n-1}}   dt+\frac{f(Q)}{\Gamma(1-\alpha)} r ^{-\alpha},\;\varepsilon\leq r\leq d ,
 $$
 $$
 (\mathfrak{D}^{\alpha}_{0+,\,\varepsilon}f)(Q)=  \frac{f(Q)}{\varepsilon^{\alpha}}  ,\; 0\leq r <\varepsilon ;
\end{equation*}
\begin{equation*}
 ( \mathfrak{D }^{\alpha}_{d-,\,\varepsilon}f)(Q)=\frac{\alpha}{\Gamma(1-\alpha)}\int\limits_{r+\varepsilon }^{d }\frac{ f (Q)- f(P+\mathbf{e}t)}{( t-r)^{\alpha +1}} dt
 +\frac{f(Q)}{\Gamma(1-\alpha)}(d-r)^{-\alpha},\;0\leq r\leq d -\varepsilon,
 $$
 $$
  ( \mathfrak{D }^{\alpha}_{d-,\,\varepsilon}f)(Q)=      \frac{ f(Q)}{\alpha} \left(\frac{1}{\varepsilon^{\alpha}}-\frac{1}{(d -r)^{\alpha} }    \right),\; d -\varepsilon <r \leq d .
 \end{equation*}
  Now, we can  define  the directional   fractional derivatives as follows
 \begin{equation*}
 \mathfrak{D }^{\alpha}_{0+}f=\lim\limits_{\stackrel{\varepsilon\rightarrow 0}{ (L_{p}) }} \mathfrak{D }^{\alpha}_{0+,\varepsilon} f  ,\;
  \mathfrak{D }^{\alpha}_{d-}f=\lim\limits_{\stackrel{\varepsilon\rightarrow 0}{ (L_{p}) }} \mathfrak{D }^{\alpha}_{d-,\varepsilon} f ,\,1\leq p\leq\infty.
\end{equation*}
The properties of these operators  are  described  in detail in the paper  \cite{firstab_lit:1kukushkin2018}. Similarly to the monograph \cite{firstab_lit:samko1987} we consider   left-side  and  right-side cases. For instance, $\mathfrak{I}^{\alpha}_{0+}$ is called  a left-side directional  fractional integral  and $ \mathfrak{D }^{\alpha}_{d-}$ is called a right-side directional fractional derivative. We suppose  $\mathfrak{I}^{0}_{0+} =I.$ Nevertheless,   this    fact can be easily proved dy virtue of  the reasonings  corresponding to the one-dimensional case and   given in \cite{firstab_lit:samko1987}. We also consider integral operators with a weighted factor (see \cite[p.175]{firstab_lit:samko1987}) defined by the following formal construction
$$
 \left(\mathfrak{I}^{\alpha}_{0+}\mu f\right) (Q  ):=\frac{1}{\Gamma(\alpha)} \int\limits^{r}_{0}
 \frac{(\mu f) (P+t\mathbf{e})}{( r-t)^{1-\alpha}}\left(\frac{t}{r}\right)^{n-1}\!\!\!\!dt,
$$
where $\mu$ is a real-valued  function.

Consider a linear combination of an uniformly elliptic operator, which is written in the divergence form, and
  a composition of a   fractional integro-differential  operator, where the fractional  differential operator is understood as the adjoint  operator  regarding  the Kipriyanov operator  (see  \cite{firstab_lit:kipriyanov1960},\cite{firstab_lit:1kipriyanov1960},\cite{kukushkin2019})
\begin{equation*}
 L  :=-  \mathcal{T}  \, +\mathfrak{I}^{\sigma}_{ 0+}\rho\, \mathfrak{D}  ^{ \alpha }_{d-},
\; \sigma\in[0,1) ,
 $$
 $$
   \mathrm{D}( L )  =H^{2}(\Omega)\cap H^{1}_{0}( \Omega ),
  \end{equation*}
where
$\,\mathcal{T}:=D_{j} ( a^{ij} D_{i}\cdot),\,i,j=1,2,...,n,$
under    the following  assumptions regarding        coefficients
\begin{equation} \label{16}
     a^{ij}(Q) \in C^{2}(\bar{\Omega}),\, \mathrm{Re} a^{ij}\xi _{i}  \xi _{j}  \geq   \gamma_{a}  |\xi|^{2} ,\,  \gamma_{a}  >0,\,\mathrm{Im }a^{ij}=0 \;(n\geq2),\,
 \rho\in L_{\infty}(\Omega).
\end{equation}
Note that in the one-dimensional case the operator $\mathfrak{I}^{\sigma }_{ 0+} \rho\, \mathfrak{D}  ^{ \alpha }_{d-}$ is reduced to   a  weighted fractional integro-differential operator  composition, which was studied properly  by many researchers (see introduction, \cite[p.175]{firstab_lit:samko1987}).
Consider a shift semigroup in a direction acting on $L_{2}(\Omega)$ and  defined as follows
$
T_{t}f(Q):=f(P+\mathbf{e}[r+ t])=f(Q+\mathbf{e}t).
$
We can formulate the following proposition.

\begin{lem}\label{L4}
The semigroup $T_{t}$ is a $C_{0}$  semigroup of contractions.
\end{lem}
 \begin{proof}
     By virtue of the continuous in average property, we conclude that $T_{t}$ is a strongly continuous semigroup. It can be easily established  due to the following reasonings, using the Minkowski inequality, we have
 $$
 \left\{\int\limits_{\Omega}|f(Q+\mathbf{e}t)-f(Q)|^{2}dQ\right\}^{\frac{1}{2}}\leq  \left\{\int\limits_{\Omega}|f(Q+\mathbf{e}t)-f_{m}(Q+\mathbf{e}t)|^{2}dQ\right\}^{\frac{1}{2}}+
 $$
 $$
 +\left\{\int\limits_{\Omega}|f(Q)-f_{m}(Q)|^{2}dQ\right\}^{\frac{1}{2}}+\left\{\int\limits_{\Omega}|f_{m}(Q)-f_{m}(Q+\mathbf{e}t)|^{2}dQ\right\}^{\frac{1}{2}}=
 $$
 $$
 =I_{1}+I_{2}+I_{3}<\varepsilon,
 $$
where $f\in L_{2}(\Omega),\,\left\{f_{n}\right\}_{1}^{\infty}\subset C_{0}^{\infty}(\Omega);$  $m$ is chosen so that $I_{1},I_{2}< \varepsilon/3 $ and $t$
is chosen so that $I_{3}< \varepsilon/3.$
Thus,  there exists such a positive  number $t_{0}$  that
$$
\|T_{t}f-f\|_{L_{2} }<\varepsilon,\,t<t_{0},
$$
for arbitrary small $\varepsilon>0.$  Using the assumption that  all functions have the zero extension outside $\Omega,$   we have
$\|T_{t}\|  \leq 1.$ Hence  we conclude that $T_{t}$ is a $C_{0}$ semigroup of contractions (see \cite{Pasy}).
\end{proof}

\begin{lem}\label{L6}
Suppose  $\rho\in \mathrm{Lip} \lambda,\,\lambda>\alpha,\,0<\alpha<1;$ then
$$
\rho\cdot\mathfrak{I}^{\alpha}_{0+}(L_{2})= \mathfrak{I}^{\alpha}_{d-}(L_{2});\;\rho\cdot\mathfrak{I}^{\alpha}_{d-}(L_{2})= \mathfrak{I}^{\alpha}_{d-}(L_{2}).
$$
 \end{lem}
 \begin{proof}
  Consider an operator
  \begin{equation}\label{7}
(\psi^{+}_{  \varepsilon }f)(Q)=  \left\{ \begin{aligned}
 \int\limits_{0}^{r-\varepsilon }\frac{ f (Q)r^{n-1}- f(T)t^{n-1}}{(  r-t)^{\alpha +1}r^{n-1}}  dt,\;\varepsilon\leq r\leq d  ,\\
   \frac{ f(Q)}{\alpha} \left(\frac{1}{\varepsilon^{\alpha}}-\frac{1}{ r ^{\alpha} }    \right),\;\;\;\;\;\;\;\;\;\;\;\;\;\;\; 0\leq r <\varepsilon,\\
\end{aligned}
 \right.
\end{equation}
where $T=P+\mathbf{e}t.$
We should prove that there exists a limit
$$
\psi^{+}_{  \varepsilon }\rho f\stackrel{L_{2}}{\longrightarrow} \psi   f,\,f\in \mathfrak{I}^{\alpha}_{0+}(L_{2}),
$$
where $\psi f$ is some function corresponding to $f.$ We have
$$
(\psi^{+}_{  \varepsilon }\rho f)(Q)=\int\limits_{0}^{r-\varepsilon }\frac{ \rho(Q)f (Q)r^{n-1}- \rho (T)f(T)t^{n-1}}{(  r-t)^{\alpha +1}r^{n-1}}  dt=
\rho(Q)\int\limits_{0}^{r-\varepsilon }\frac{ f (Q)r^{n-1}-  f(T)t^{n-1}}{(  r-t)^{\alpha +1}r^{n-1}}  dt+
$$
$$
+\int\limits_{0}^{r-\varepsilon }\frac{   f(T)[\rho(Q)-\rho(T)]}{(  r-t)^{\alpha +1}} \left(\frac{t}{r }\right)^{n-1}  \!\!\!dt=A_{\varepsilon}(Q)+B_{\varepsilon}(Q)
 ,\;\varepsilon\leq r\leq d;
$$
$$
(\psi^{+}_{  \varepsilon }\rho f)(Q)= \rho(Q)f(Q)  \frac{1}{\alpha}\left(\frac{1}{\varepsilon^{\alpha}}-\frac{1}{ r ^{\alpha} }    \right),\;  0\leq r <\varepsilon.
$$
Hence, we get
$$
\|\psi^{+}_{  \varepsilon_{n+1} }\rho f-\psi^{+}_{  \varepsilon_{n} }\rho f\|_{L_{2}(\Omega)}\leq \|\psi^{+}_{  \varepsilon_{n+1} }\rho f-\psi^{+}_{  \varepsilon_{n} }\rho f\|_{L_{2}(\Omega')}
 +\|\psi^{+}_{  \varepsilon_{n+1} }\rho f-\psi^{+}_{  \varepsilon_{n} }\rho f\|_{L_{2}(\Omega_{n})},
$$
where $\{\varepsilon_{n}\}_{1}^{\infty}\subset \mathbb{R}_{+}$ is a strictly decreasing sequence that is chosen in an arbitrary way, $\Omega_{n}:=  \omega\times \{0<r<\varepsilon_{n}\}, $
 $\Omega':=  \Omega\setminus \Omega_{n} .$
It is clear that
$$
\|A_{\varepsilon_{n+1}}-A_{\varepsilon_{n}}\|_{L_{2}(\Omega')}\leq \|\rho\|_{L_{\infty}(\Omega)}\|\psi^{+}_{  \varepsilon_{n+1} }f-\psi^{+}_{  \varepsilon_{n} }f\|_{L_{2}(\Omega')},
$$
Since in accordance with Theorem 2.3 \cite{firstab_lit:1kukushkin2018}  the sequence $ \psi^{+}_{  \varepsilon_{n} }f,\,(n=1,2,...) $ is fundamental for the defined function $f,$ with respect to the $L_{2}(\Omega)$ norm,    then the sequence  $ A_{\varepsilon_{n}} $ is also fundamental with respect to the $L_{2}(\Omega')$ norm.
 Having used the H\"{o}lder properties of $\rho,$   we have
$$
\|B_{\varepsilon_{n+1}}-B_{\varepsilon_{n}}\|_{L_{2}(\Omega')}\leq M \left\{\int\limits_{\Omega'}\left(\int\limits_{r-\varepsilon_{n}}^{r-\varepsilon_{n+1}}\frac{   |f(T)| }{(  r-t)^{\alpha +1-\lambda}} \left(\frac{t}{r }\right)^{n-1}  \!\!\!dt\right)^{2}dQ\right\}^{\frac{1}{2}}.
$$
Note that applying Theorem 2.3 \cite{firstab_lit:1kukushkin2018}, we have
$$
\left\{\int\limits_{\Omega}\left(\int\limits_{0}^{r}\frac{   |f(T)| }{(  r-t)^{\alpha +1-\lambda}} \left(\frac{t}{r }\right)^{n-1}  \!\!\!dt\right)^{2}dQ\right\}^{\frac{1}{2}}\leq C \|f\|_{L_{2}}.
$$
Hence the sequence $ \left\{B_{\varepsilon_{n}}\right\}_{1}^{\infty} $ is fundamental with respect to the $L_{2}(\Omega')$ norm.
Therefore
$$
\|\psi^{+}_{  \varepsilon_{n+1} }\rho f-\psi^{+}_{  \varepsilon_{n} }\rho f\|_{L_{2}(\Omega')}\rightarrow 0,\,n\rightarrow \infty.
$$
Consider
$$
\|\psi^{+}_{  \varepsilon_{n+1} }\rho f-\psi^{+}_{  \varepsilon_{n} }\rho f\|_{L_{2}(\Omega_{n})}\leq
\|\psi^{+}_{  \varepsilon_{n+1} }\rho f-\psi^{+}_{  \varepsilon_{n} }\rho f\|_{L_{2}(\Omega_{n+1})}+
$$
$$
+\left\{\int\limits_{\omega}d\chi  \int\limits_{\varepsilon_{n+1}}^{\varepsilon_{n}}
|A_{\varepsilon_{n+1}}(Q)+B_{\varepsilon_{n+1}}(Q)|^{2}rdr\right\}^{\frac{1}{2}}+
$$
$$
 +\frac{1}{\alpha}\left\{\int\limits_{\omega}d\chi  \int\limits_{\varepsilon_{n+1}}^{\varepsilon_{n}}
\left|\rho(Q)f(Q)  \left(\frac{1}{\varepsilon_{n}^{\alpha}}-\frac{1}{ r ^{\alpha} }    \right)\right|^{2}rdr\right\}^{\frac{1}{2}}=
I_{1}+I_{2}+I_{3}.
$$
We have
$$
I_{1}\leq \frac{1}{\alpha}\left(\frac{1}{\varepsilon_{n}^{\alpha}}-\frac{1}{\varepsilon_{n+1}^{\alpha}   }    \right)\|\rho\|_{L_{\infty}} \int\limits_{\omega}d\chi\int\limits_{0}^{\varepsilon_{n+1}}
   f(Q)r dr\leq
$$
$$
\leq\frac{1}{\alpha}\left(\frac{1}{\varepsilon_{n}^{\alpha}}-\frac{1}{\varepsilon_{n+1}^{\alpha}   }    \right)\|\rho\|_{L_{\infty}}  \int\limits_{\omega}\left\{\int\limits_{0}^{\varepsilon_{n+1}}
  |f(Q)|^{2}r dr \right\}^{\frac{1}{2}}\left\{\int\limits_{0}^{\varepsilon_{n+1}}
   r dr \right\}^{\frac{1}{2}}d\chi\leq
$$
$$
\leq\frac{1}{\sqrt{2}\alpha}\left(\frac{1}{\varepsilon_{n}^{\alpha}}-\frac{1}{\varepsilon_{n+1}^{\alpha}   }    \right) \varepsilon_{n+1}\|\rho\|_{L_{\infty}} \|f\|_{L_{2}}.
$$
Hence $I_{1}\rightarrow 0,\,n\rightarrow \infty.$
Using the  estimates used above, it is not hard to prove that $I_{2},I_{3}\rightarrow 0,\,n\rightarrow \infty.$ The proof is left to a reader.
Therefore
$$
\|\psi^{+}_{  \varepsilon_{n+1} }\rho f-\psi^{+}_{  \varepsilon_{n} }\rho f\|_{L_{2}(\Omega_{n})}\rightarrow 0,\,n\rightarrow \infty.
$$
Combining the obtained results, we have
$$
\|\psi^{+}_{  \varepsilon_{n+1} }\rho f-\psi^{+}_{  \varepsilon_{n} }\rho f\|_{L_{2}(\Omega)}\rightarrow 0,\,n\rightarrow \infty.
$$
Using  Theorem 2.2 \cite{firstab_lit:1kukushkin2018}, we obtain the desired result for the case corresponding  to the class $\mathfrak{I}^{\alpha}_{0+}(L_{2}).$ The proof corresponding to the class $\mathfrak{I}^{\alpha}_{d-}(L_{2})$ is absolutely analogous.
\end{proof}

The following theorem is formulated in terms of the infinitesimal  generator $-A$ of the semigroup $T_{t}.$

\begin{teo}\label{T3} We claim that  $L=Z^{\alpha}_{G,F}(A).$ Moreover  if  $ \gamma_{a} $ is sufficiently large in comparison  with $\|\rho\|_{L_{\infty}},$ then $L$ satisfies conditions H1-H2, where we put $\mathfrak{M}:=C_{0}^{\infty}(\Omega),$   if we additionally assume that $\rho \in \mathrm{Lip}\lambda,\,   \lambda>\alpha  ,$ then    $ \tilde{\mathcal{H}}=H.$
\end{teo}
\begin{proof}
     By virtue of   Corollary 3.6 \cite[p.11]{Pasy}, we have
\begin{equation} \label{17}
\|(\lambda+A)^{-1}\|  \leq \frac{1}{\mathrm{Re} \lambda },\,\mathrm{Re}\lambda>0.
\end{equation}
Inequality \eqref{17} implies that $A$ is m-accretive.
Using formula \eqref{9},  we can define positive fractional powers $\alpha\in (0,1)$ of the operator $A. $
Applying  the Balakrishnan formula, we obtain
 \begin{equation}\label{18}
A^{\alpha}f:=\frac{\sin\alpha \pi}{\pi}\int\limits_{0}^{\infty}\lambda^{\alpha-1}(\lambda  +A)^{-1} A f\,d \lambda=\frac{1}{\Gamma(-\alpha)}\int\limits_{0}^{\infty}\frac{T_{t}-I}{t^{\alpha+1}}fdt,\,f\in \mathrm{D}(A).
\end{equation}
Hence, in the concrete  form of writing we have
 \begin{equation}\label{18.0.1}
A^{\alpha}f(Q)=\frac{1}{\Gamma(-\alpha)}\int\limits_{0}^{\infty}\frac{f(Q+\mathbf{e}t)-f(Q)}{t^{\alpha+1}}dt=
$$
$$
=\frac{\alpha}{\Gamma(1-\alpha)}\int\limits_{r}^{d   (\mathbf{e})  }\frac{f(Q)-f(P+\mathbf{e}t)}{(t-r)^{\alpha+1}}dt+ \frac{f (Q)}{\Gamma(1-\alpha)} \{d(\mathbf{e})-r\}^{-\alpha} =
\mathfrak{D}^{\alpha}_{d-}f(Q),\ f\in  \mathrm{D}  (A),
\end{equation}
where $d(\mathbf{e})$ is the distance from the point $P$ to the edge of $\Omega$ along the direction $\mathbf{e}.$ Note  that a relation  between positive  fractional powers of the operator $A$    and the  Riemann-Liouville  fractional derivative was demonstrated  in the one-dimensional case   in the paper    \cite{firstab_lit:1Ashyralyev}.

 Consider a restriction   $A_{0}\subset A,\,\mathrm{D}(A_{0})=C^{\infty}_{0}( \Omega )$ of the operator  $A.$ Note that,   since  the infinitesimal  generator  $-A$ is a closed operator (see \cite{Pasy}), then $A_{0}$  is closeable.
It is not hard to prove that $ \tilde{A}_{0}$ is an m-accretive operator. For this purpose, note that   since the operator $A$ is  m-accretive, then  by virtue of  \eqref{6}, we get
\begin{equation*} \mathrm{Re}(\tilde{A}_{0} f,f)_{ \mathfrak{H}}\geq0 ,\,f\in \mathrm{D}(\tilde{A}_{0}).
\end{equation*}
This gives us an opportunity to conclude that
$$
   \|f\|^{2}_{\mathfrak{H}}\leq  \frac{1}{ t^{2}} \left\{ \| \tilde{A}_{0} f\|^{2}_{ \mathfrak{H}}+2t \mathrm{Re}(\tilde{A}_{0}f,f)_{ \mathfrak{H}}+t^{2}\|   f \|^{2}_{ \mathfrak{H}}\right\}  ;\,\|f\|^{2}_{\mathfrak{H}}\leq  \frac{1}{t^{2}}  \|(\tilde{A}_{0}+t)f \|^{2}_{ \mathfrak{H}},\,t>0.
$$
Therefore
\begin{equation*}
\|(\tilde{A}_{0}+t)^{-1}\|_{\mathrm{R} \rightarrow \mathfrak{H}}\leq\frac{1}{ t},\,t>0,
\end{equation*}
where   $\mathrm{R}:=\mathrm{R}(\tilde{A}_{0}+t).$ Hence, in accordance with Lemma \ref{L1}, we obtain that the operator $\tilde{A}_{0}$ is m-accretive. Since there does not exist an accretive extension  of an m-accretive operator (see \cite[p.279]{firstab_lit:kato1980} ) and $\tilde{A}_{0}\subset A,$ then $\tilde{A}_{0}= A.$
It is easy to prove  that
\begin{equation}\label{26.2}
\|Af\|_{L_{2}}\leq C\|f\|_{H_{0}^{1}},\,f\in H_{0}^{1}(\Omega),
\end{equation}
 for this purpose we should establish  a representation $Af(Q) =-(\nabla f ,\mathbf{e})_{\mathbb{E}^{n}}, f\in C^{\infty}_{0}(\Omega)$  the rest of the proof   is  left to a reader.
Thus, we get   $H_{0}^{1}(\Omega) \subset \mathrm{D}(A),$ and as a result $A^{\alpha}f = \mathfrak{D}^{\alpha}_{d-}f ,\ f\in  H_{0}^{1}(\Omega) .$
Let us find a representation for the  operator $G.$
Consider an operator
 $$
 Bf(Q)=\!\int_{0}^{r}\!\!f(P+\mathbf{e}[r-t])dt,\,f\in L_{2}(\Omega).
 $$It is not hard to prove that  $B\in \mathcal{B}(L_{2}),$   applying the generalized Minkowski inequality, we get
 $$
 \|Bf\|_{L_{2} }\leq \int\limits_{0}^{\mathrm{diam\,\Omega}}dt  \left(\int\limits_{\Omega}|f(P+\mathbf{e}[r-t])|dQ\right)^{1/2}\leq C\|f\|_{L_{2} }.
 $$
    The fact $A^{-1}_{0}\subset B$   follows from properties of the  one-dimensional integral defined on smooth functions.
It is  a  well-known fact (see Theorem 2 \cite[p.555]{firstab_lit:Smirnov5}) that  since  $A_{0}$ is closeable and there exists a bounded operator $ A^{-1}_{0},$ then there exists a bounded operator $A^{-1}=\tilde{A}^{-1}_{0}=  \widetilde{A ^{-1}_{0}}  .$ Using  this relation we conclude that $A^{-1}\subset B.$ It is obvious that
\begin{equation}\label{24}
\int\limits_{\Omega }  A\left(B  \mathcal{T}  f \cdot g\right) dQ =\int\limits_{\Omega } AB\mathcal{T}f \cdot g\,  dQ +\int\limits_{\Omega }  B \mathcal{T} f \cdot Ag  \,dQ ,\, f\in C^{2}(\bar{\Omega}),g\in C^{\infty}_{0}( \Omega ).
\end{equation}
Using the   divergence  theorem, we get
\begin{equation}\label{25}
\int\limits_{\Omega }  A\left(B\mathcal{T}f \cdot g\right) \,dQ=\int\limits_{S}(\mathbf{e},\mathbf{n})_{\mathbb{E}^{n}}(B\mathcal{T}f\cdot  g)(\sigma)d\sigma,
\end{equation}
where $S$ is the surface of $\Omega.$
Taking into account   that $ g(S)=0$ and combining   \eqref{24},\eqref{25}, we get
\begin{equation}\label{26}
    -\int\limits_{\Omega }  AB\mathcal{T} f\cdot \bar{g} \, dQ=  \int\limits_{\Omega } B\mathcal{T} f\cdot   \overline{A  g} \, dQ,\, f\in C^{2}(\bar{\Omega}),g\in C^{\infty}_{0}( \Omega ).
\end{equation}
Suppose that $f\in H^{2}(\Omega),$ then    there exists a sequence $\{f_{n}\}_{1}^{\infty}\subset C^{2}(\bar{\Omega})$ such that
$ f_{n}\stackrel{ H^{2}}{\longrightarrow}  f$ (see \cite[p.346]{firstab_lit:Smirnov5}).  Using this fact, it is not hard to prove that
$\mathcal{T}f_{n}\stackrel{L_{2}}{\longrightarrow} \mathcal{T}f.$  Therefore  $AB\mathcal{T}f_{n}\stackrel{L_{2}}{\longrightarrow} \mathcal{T}f,$  since  $AB\mathcal{T}f_{n}=\mathcal{T}f_{n}.$ It is also clear that   $B\mathcal{T}f_{n}\stackrel{L_{2}}{\longrightarrow} B\mathcal{T}f,$ since $B$ is continuous.
Using these facts, we can extend relation \eqref{26} to the following
\begin{equation}\label{26.1}
  -\int\limits_{\Omega }  \mathcal{T} f \cdot  \bar{g} \, dQ= \int\limits_{\Omega } B\mathcal{T}f\,  \overline{Ag} \, dQ,\; f\in \mathrm{D}(L),\,g\in C_{0}^{\infty}(\Omega).
\end{equation}
It was previously proved that $H_{0}^{1}(\Omega) \subset \mathrm{D}(A),\, A^{-1} \subset B.$ Hence $G Af=B\mathcal{T} f,\, f\in  \mathrm{D}(L),$  where $G:=B\mathcal{T}B.$  Using  this fact    we can rewrite relation \eqref{26.1} in a   form
\begin{equation}\label{26.102}
  -\int\limits_{\Omega }  \mathcal{T} f \cdot  \bar{g} \, dQ= \int\limits_{\Omega } G Af\,  \overline{Ag} \, dQ,\; f\in \mathrm{D}(L),\,g\in C_{0}^{\infty}(\Omega).
\end{equation}
Note that   in accordance with the fact $A=\tilde{A}_{0},$ we have
$$
\forall g\in \mathrm{D}(A),\,\exists \{g_{n}\}_{1}^{\infty}\subset C^{\infty}_{0}( \Omega ),\,    g_{n}\xrightarrow[      A       ]{}g.
$$
Therefore, we can extend   relation \eqref{26.102} to the following
 \begin{equation}\label{27}
     -\int\limits_{\Omega }  \mathcal{T} f \cdot  \bar{g} \, dQ= \int\limits_{\Omega } G Af \, \overline{Ag} \, dQ,\; f\in \mathrm{D}(L),\,g\in \mathrm{D}(A).
\end{equation}
Relation \eqref{27} indicates that $G Af\in \mathrm{D}( A ^{\ast})$   and it is clear that $  -\mathcal{T}\subset A ^{\ast}GA.$ On the other hand in accordance with  Chapter VI, Theorem 1.2    \cite{firstab_lit: Berezansk1968}, we have that $-\mathcal{T}$ is a closed operator, hence in accordance with Lemma \ref{L1} the operator  $-\mathcal{T}$ is m-accretive. Therefore $-\mathcal{T}= A  ^{\ast}GA,$ since $A ^{\ast}GA$ is accretive. Note that  by virtue of  Theorem 2.1 \cite{firstab_lit:1kukushkin2018}, we have     $(\mathfrak{I}^{\sigma }_{0+}\rho\, \cdot)\in \mathcal{B}(L_{2}).$
    It was previously proved that $\mathfrak{D}^{\alpha}_{d-}f= A^{\alpha}f,\,f\in H^{1}_{0}(\Omega).$         Thus, the  representation  $L=Z^{\alpha}_{GF}(A),$     where $G:=B\mathcal{T}B,\,F:=(\mathfrak{I}^{\sigma }_{0+}\rho\,\cdot)$   has been established.

 Let us prove that the operator $L$ satisfy   conditions H1--H2.    Choose the space  $L_{2}(\Omega)$ as a space $\mathfrak{H},$ the set  $C_{0}^{\infty}(\Omega)$ as a linear  manifold $\mathfrak{M},$ and the space  $H^{1}_{0}(\Omega)$ as a space $\mathfrak{H}_{+}.$ By virtue of the  Rellich-Kondrashov theorem, we have  $H_{0}^{1}(\Omega)\subset\subset L_{2}(\Omega).$  Thus, condition  H1  is fulfilled.
Using   simple  reasonings, we come to the following inequality
 \begin{equation}\label{27.1.0}
\left|\int\limits_{\Omega }  \mathcal{T} f \cdot  \bar{g} \, dQ\right|\leq C\|f\|_{H_{0}^{1}}\|g\|_{H_{0}^{1}},\; f,g\in C_{0}^{\infty}(\Omega).
\end{equation}
     Let us prove that
\begin{equation}\label{32.1}
 |\left( \mathfrak{I}^{\sigma }_{0+}\rho\,\mathfrak{D}^{\alpha}_{d-}  f,g\right)_{ L _{2 }}| \leq K\|f\|_{H_{0}^{1}}\|g\|_{L_{2}},\,f,g\in C^{\infty}_{0}(\Omega),
\end{equation}
where $K=C\|\rho\|_{L_{\infty}}.$
Using  a fact that    the operator $(\mathfrak{I}^{\sigma }_{0+}\rho\,\cdot)$ is bounded, we obtain
\begin{equation}\label{32.15}
\|\mathfrak{I}^{\sigma }_{0+}\rho\,\mathfrak{D}^{\alpha}_{d-}  f\|_{L_{2}}\leq C \|\rho\|_{L_{\infty}} \| \mathfrak{D}^{\alpha}_{d-}  f\|_{L_{2}},\, f\in C_{0}^{\infty}(\Omega).  \end{equation}
Taking into account that $ A ^{-1} $ is bounded, $A$ is m-accretive, applying  Lemma \ref{L3.1} analogously to  \eqref{15.0.1}, we conclude that
$
\|A^{\alpha} f\|_{L_{2}}\leq C \|Af\|_{L_{2}},\,f\in \mathrm{D}(A).
$
Using \eqref{18.0.1},\eqref{26.2}, we get
$
\| \mathfrak{D}^{\alpha}_{d-}  f\|_{L_{2}}\leq C\|f\|_{H_{0}^{1}},\,f\in C_{0}^{\infty}(\Omega).
$
Combining this relation with \eqref{32.15}, we obtain
$$
\|\mathfrak{I}^{\sigma }_{0+}\rho\,\mathfrak{D}^{\alpha}_{d-}  f\|_{L_{2}}\leq K \|f\|_{H_{0}^{1}},\,f\in C_{0}^{\infty}(\Omega).
$$
Using this inequality, we can easily obtain \eqref{32.1}, from what follows that
$$
\mathrm{Re}(\mathfrak{I}^{\sigma }_{0+}\rho\,\mathfrak{D}^{\alpha}_{d-}  f,f)_{L_{2}}\geq -K\|f\|^{2}_{H^{1}_{0}},\, f \in C^{\infty}_{0}(\Omega).
$$
On the other hand, using a uniformly elliptic property of the operator $\mathcal{T}$ it is not hard to prove that
\begin{equation}\label{27.2}
-\mathrm{Re}(\mathcal{T} f,f)\geq \gamma_{a}\|f\|_{H_{0}^{1}},\,f\in C^{\infty}_{0}(\Omega),
\end{equation}
the proof of this fact  is obvious and left to a reader (see  \cite{firstab_lit:1kukushkin2018}). Now, if we assume that $\gamma_{a}>K,$ then we obtain the fulfillment of condition
H2.

  Assume additionally that $\rho \in \mathrm{Lip}\lambda,\,\lambda>\alpha,$ let us prove that  $C_{0}^{\infty}(\Omega)\subset \mathrm{D}(L^{\ast}).$
Note that
 $$
 \int\limits_{\Omega}D_{j} ( a^{ij} D_{i}f)\,gdQ=\int\limits_{\Omega}f\,\overline{ D_{j} ( a^{ji} D_{i}g)}dQ,\,f\in \mathrm{D}(L),\, g\in C_{0}^{\infty}(\Omega).
 $$
Using this equality, we conclude that $(-\mathcal{T})^{\ast}$ is defined on $C_{0}^{\infty}(\Omega).$ Applying   the Fubini theorem,   Lemma \ref{L6}, Lemma 2.6 \cite{firstab_lit:1kukushkin2018},     we get
\begin{equation*}
 \left( \mathfrak{I}^{\sigma}_{0+}\rho\,\mathfrak{D}^{\alpha}_{d-}  f,g\right)_{L_{2}}= \left( \mathfrak{D}^{\alpha}_{d-}  f,\rho\,\mathfrak{I}^{\sigma}_{d-}g\right)_{L_{2}}= \left(   f,\mathfrak{D}^{\alpha}_{0+}\rho\,\mathfrak{I}^{\sigma}_{d-}g\right)_{L_{2}}\!\!, \,f\in \mathrm{D}(L),\, g\in C_{0}^{\infty}(\Omega).
\end{equation*}
Therefore the operator
$
\left(\mathfrak{I}^{\sigma}_{0+}\rho\,\mathfrak{D}^{\alpha}_{d-}\right)^{\ast}
$
is defined on $C_{0}^{\infty}(\Omega).$
Taking into account the above reasonings, we conclude that    $ C_{0}^{\infty}(\Omega) \subset \mathrm{D}( L^{\ast} ).$ Combining  this fact with relation H2, we obtain     $\tilde{\mathcal{H}}=H$ (see Remark \ref{R1}).
\end{proof}

\begin{corol}\label{C2}
 Consider  a one-dimensional case,     we claim that  $L\in \mathfrak{G_{\alpha}}.$
\end{corol}
\begin{proof}
It is not hard to prove   that $ \| A_{0}f \|_{L_{2}}=\|f\|_{H_{0}^{1}},\,f\in C_{0}^{\infty}(\Omega).$ This relation can be extended to the following
\begin{equation}\label{37}
 \| A f \|_{L_{2}}=\|f\|_{H_{0}^{1}},\,f\in H_{0}^{1}(\Omega),
\end{equation}
whence $\mathrm{D}(A)=H_{0}^{1}(\Omega).$
 Taking into account  the Rellich-Kondrashov theorem, we conclude that  $ A ^{-1}$ is compact. Thus, to show that conditions of Theorem \ref{T1} are fulfilled  we need  prove  that  the  operator $G:=B\mathcal{T}B$ is bounded and $\mathrm{R}( A )\subset \mathrm{D}(G).$ We can establish  the following relation  by direct calculations
$
GA_{0}f =B\mathcal{T}f =a^{11} A_{0}f ,\,f\in C_{0}^{\infty}(\Omega),
$
where $a^{11}=a^{ij},\;i,j=1.$ Using this equality, we can easily prove that
$
\|G  A f\|_{L_{2}}\leq C\| A f\|_{L_{2}},\, f\in \mathrm{D}( A ).
$
Thus, we obtain the desired  result.
\end{proof}

\vspace{0.5cm}

\noindent{\bf  Riesz potential}\\

 Consider a   space $L_{2}(\Omega),\,\Omega:=(-\infty,\infty).$      We denote by $H^{2,\,\lambda}_{0}(\Omega)$ the completion of the set  $C^{\infty}_{0}(\Omega)$  with the norm
$$
\|f\|_{H^{2,\lambda}_{0}}=\left\{\|f\|^{2}_{L_{2}(\Omega) }+\|f''\|^{2}_{L_{2}(\Omega,\omega^{\lambda})} \right\}^{1/2},\,  \lambda\in \mathbb{R},
$$
where $\omega(x):=  (1+|x|).$
Let us notice the following fact  (see Theorem 1 \cite{firstab_lit:1Adams}), if $\lambda>4,$ then
$
H^{2,\,\lambda}_{0}(\Omega)\subset\subset L_{2}(\Omega).
$
Consider a Riesz potential
$$
I^{\alpha}f(x)=B_{\alpha}\int\limits_{-\infty}^{\infty}f (s)|s-x|^{\alpha-1} ds,\,B_{\alpha}=\frac{1}{2\Gamma(\alpha)  \cos  \alpha \pi / 2   },\,\alpha\in (0,1),
$$
where $f$ is in $L_{p}(-\infty,\infty),\,1\leq p<1/\alpha.$
It is  obvious that
$
I^{\alpha}f= B_{\alpha}\Gamma(\alpha) (I^{\alpha}_{+}f+I^{\alpha}_{-}f),
$
where
$$
I^{\alpha}_{\pm}f(x)=\frac{1}{\Gamma(\alpha)}\int\limits_{0}^{\infty}f (s\pm x) s ^{\alpha-1} ds,
$$
the last operators are known as fractional integrals on a whole  real axis   (see \cite[p.94]{firstab_lit:samko1987}). Assume that the following  condition holds
 $ \sigma/2 + 3/4<\alpha<1 ,$ where $\sigma$ is a non-negative  constant. Following the idea of the   monograph \cite[p.176]{firstab_lit:samko1987}
 consider a sum of a differential operator and  a composition of    fractional integro-differential operators
$$
 L    := \tilde{\mathcal{T}}   +I^{\sigma}_{+}\,\rho \,  I^{2(1-\alpha)}\frac{d^{2}}{dx^{2}}  \,,
 $$
where  $$
\mathcal{T} := \frac{d^{2}}{dx^{2}}\left(a  \frac{d^{2}}{dx^{2}}  \cdot \right) ,\, \mathrm{D}(\mathcal{T})=C^{\infty}_{0}(\Omega) ,
$$
$$
\, \rho(x)\in L_{\infty}(\Omega) ,\,a(x)\in L_{\infty}(\Omega)\cap C^{ 2 }( \Omega ),\, \mathrm{Re}\,a(x) >\gamma_{a}(1+|x|)^{5},\,\gamma_{a}>0.
$$
Consider a family of operators
$$
T_{t}f(x)=(2\pi t)^{-1/2}\int\limits_{-\infty}^{\infty}e^{-(x-\tau)^{2}/2t}f(\tau)d\tau,\;t>0,\;
 T_{t}f(x)=f(x),\;t=0,\,f\in L_{2}(\Omega).
$$
\begin{lem}\label{L6.1}
$T_{t}$ is a $C_{0}$ semigroup of contractions.
\end{lem}
\begin{proof}
Let us establish the semigroup property,  by definition  we have $T_{0}=I.$ Consider the following formula, note that the interchange of the integration order can be easily substantiated
$$
T_{t}T_{t'}f(x)=\frac{1}{\sqrt{2\pi t}\sqrt{2\pi t'}}
\int\limits_{-\infty}^{\infty} e^{-\frac{(x-u)^{2}}{2t}}du\int\limits_{-\infty}^{\infty} e^{-\frac{(u-\tau)^{2}}{2t}}f(\tau)d\tau=
$$
$$
 =\frac{1}{\sqrt{2\pi t}\sqrt{2\pi t'}}
\int\limits_{-\infty}^{\infty} f(\tau)d\tau \int\limits_{-\infty}^{\infty}e^{-\frac{(x-u)^{2}}{2t}} e^{-\frac{(u-\tau)^{2}}{2t}}  du=\frac{1}{\sqrt{2\pi t}\sqrt{2\pi t'}}
\int\limits_{-\infty}^{\infty} f(\tau)d\tau \int\limits_{-\infty}^{\infty}e^{-\frac{(x-v-\tau)^{2}}{2t}} e^{-\frac{  v ^{2}}{2t}}  dv.
$$
On the other hand, in accordance with the formula \cite[p.325]{firstab_lit:Yosida}, we have
$$
\frac{1}{\sqrt{2\pi (t+t')} }e^{-\frac{(x-\tau)^{2}}{2t}}=\frac{1}{\sqrt{2\pi t}\sqrt{2\pi t'}}\int\limits_{-\infty}^{\infty}e^{-\frac{(x-\tau-v)^{2}}{2t}} e^{-\frac{ v ^{2}}{2t}}  dv.
$$
Hence
$$
\frac{1}{\sqrt{2\pi (t+t')} }\int\limits_{-\infty}^{\infty}e^{-\frac{(x-\tau)^{2}}{2t}}f(\tau) d\tau=\frac{1}{\sqrt{2\pi t}\sqrt{2\pi t'}}
\int\limits_{-\infty}^{\infty} f(\tau)d\tau \int\limits_{-\infty}^{\infty}e^{-\frac{(x-v-\tau)^{2}}{2t}} e^{-\frac{  v ^{2}}{2t}}  dv,
$$
from what immediately  follows  the fact   $T_{t}T_{t'}f=T_{t+t'}f.$
Let us show that $T_{t}$ is a $C_{0}$ semigroup of contractions.
Observe that
$$
(2\pi t)^{-1/2}\int\limits_{-\infty}^{\infty}e^{- \tau ^{2}/2t} d\tau =1.
$$
Therefore,  using the generalized Minkowski inequality (see (1.33) \cite[p.9]{firstab_lit:samko1987}), we get
$$
\|T_{t}f\|_{L_{2}}=\left(\int\limits_{-\infty}^{\infty} \left|\int\limits_{-\infty}^{\infty}f(x+s) N_{t}(s)ds  \right|^{2}dx \right)^{1/2}  \leq
  $$
  $$
  \leq\int\limits_{-\infty}^{\infty} N_{t}(s)ds \left(\int\limits_{-\infty}^{\infty}\left|f(x+s)   \right|^{2}dx \right)^{1/2}=\|f\|_{L_{2}},\,f\in C_{0}^{\infty}(\Omega),
$$
where $N_{t}(x):=(2\pi t)^{-1/2}e^{-x^{2}/2t}.$
It is clear that the last inequality can be extended to $L_{2}(\Omega),$ since $C_{0}^{\infty}(\Omega)$ is dense in $L_{2}(\Omega).$
Thus, we conclude that $T_{t}$ is a   semigroup of contractions.

Let us establish a strongly continuous property.  Assuming that  $z=(x-\tau)/\sqrt{t},$  we get in an obvious way
$$
\|T_{t}f-f\|_{L_{2}}=\left(\int\limits_{-\infty}^{\infty}\left|\int\limits_{-\infty}^{\infty}N_{1}(z)\left[ f(x-\sqrt{t}z)-f( x)\right] dz \right|^{2} dx\right)^{1/2}\!\!\!\leq
$$
$$
\leq\int\limits_{-\infty}^{\infty}N_{1}(z)\left(\int\limits_{-\infty}^{\infty}\left[ f(x-\sqrt{t}z)-f(x)\right]^{2} dx \right)^{1/2} \!\!\!dz,\, f\in L_{2}(\Omega),
$$
where $N_{1}=N_{t}|_{t=1}.$
Observe that, for arbitrary fixed $t,z$ we have
$$
N_{1}(z)\left(\int\limits_{-\infty}^{\infty}\left[ f(x-\sqrt{t}z)-f(x)\right]^{2} dx \right)^{1/2}\!\!\!\leq
 $$
 $$
 \leq N_{1}(z)\left(\int\limits_{-\infty}^{\infty}\left[ f(x-\sqrt{t}z) \right]^{2} dx \right)^{1/2}+N_{1}(z)\|f\|_{L_{2}}\leq 2 N_{1}(z)\|f\|_{L_{2}}.
$$
Applying the Fatou--Lebesgue theorem, we get
$$
 \overline{\lim\limits_{ t\rightarrow 0}}\int\limits_{\infty}^{\infty}N_{1}(z)\left(\int\limits_{\infty}^{\infty}\left[ f(x-\sqrt{t}z)-f(x)\right]^{2} dx \right)^{1/2} \!\!\!dz\leq
 \!\!\int\limits_{\infty}^{\infty}N_{1}(z)\,\overline{\lim\limits_{ t\rightarrow 0}}\left(\int\limits_{\infty}^{\infty}\left[ f(x-\sqrt{t}z)-f(x)\right]^{2} dx \right)^{\!\!1/2}\!\!\!\!=0,
$$
from what follows that $\|T_{t}f-f\|_{L_{2}}\rightarrow 0,\;t\rightarrow  0.$
Hence $T_{t}$ is a $C_{0}$ semigroup of contractions.

\end{proof}

 The following theorem is formulated in terms of     the infinitesimal generator   $-A$ of the semigroup $T_{t}.$

\begin{teo}\label{T4}  We claim that $L =Z^{\alpha}_{G,F}(A).$ Moreover,  if    $\min\{\gamma_{a},\delta\},\,(\delta>0)$ is  sufficiently large in comparison with   $\|\rho\|_{L_{\infty}},$ then
a  perturbation $L+\delta I$ satisfies  conditions  H1-H2,  where we put $\mathfrak{M}:=C_{0}^{\infty}(\Omega).$
 \end{teo}
\begin{proof}
Let us prove that
\begin{equation*}\label{37.712}
Af =-\frac{1}{2}\frac{d^{\,2}f }{dx^{2}}\;\mathrm{a.e.},\,f\in \mathrm{D}(A).
\end{equation*}
Consider an operator $J_{n}=n(nI +A)^{-1}.$ It is clear that $AJ_{n}=n(I-J_{n}).$ Using the formula
$$
(nI+A)^{-1}f=\int\limits_{0}^{\infty}e^{-n t}T_{t}f dt,\,n>0,\,f\in L_{2}(\Omega),
$$
 we   easily obtain
$$
J_{n}f(x)=\frac{n}{\sqrt{2\pi }}\int\limits_{0}^{\infty}e^{-n t} t^{-1/2}  dt
\int\limits_{-\infty}^{\infty} e^{-\frac{(x-\tau)^{2}}{2t}}f(\tau)d\tau=\frac{n}{\sqrt{2\pi  }}\int\limits_{-\infty}^{\infty}f(\tau)d\tau
\int\limits_{0}^{\infty}   e^{-nt-\frac{(x-\tau)^{2}}{2t}}t^{-1/2} dt=
$$
$$
 =  \sqrt{\frac{2n}{\pi}} \int\limits_{-\infty}^{\infty}f(\tau)d\tau
\int\limits_{0}^{\infty}   e^{-\sigma^{2}-n\frac{(x-\tau)^{2}}{2\sigma^{2}}}d\sigma,\;t= \sigma^{2}/n.
$$
Applying  the  following   formula (see (3)  \cite[p.336]{firstab_lit:Yosida})
\begin{equation}\label{37.012}
\int\limits_{0}^{\infty}   e^{-(\sigma^{2}+  c^{2}/  \sigma^{2} )}d\sigma=\frac{\sqrt{\pi}}{2}e^{-2 |c| },
\end{equation}
we obtain
$$
 J_{n}f(x)=\sqrt{\frac{ n}{2}} \int\limits_{-\infty}^{\infty}f(\tau) e^{- \sqrt{2n}|x-\tau| }  d\tau = \sqrt{\frac{ n}{2}} \int\limits_{-\infty}^{x}f(\tau) e^{- \sqrt{2n} (x-\tau)  }  d\tau+\sqrt{\frac{ n}{2}} \int\limits_{x}^{\infty}f(\tau) e^{- \sqrt{2n} (\tau-x) }  d\tau=
 $$
 $$
=\sqrt{\frac{ n}{2}}e^{- \sqrt{2n} x} \int\limits_{-\infty}^{x}f(\tau) e^{ \sqrt{2n}   \tau   }  d\tau+\sqrt{\frac{ n}{2}}e^{ \sqrt{2n} x} \int\limits_{x}^{\infty}f(\tau) e^{ -\sqrt{2n}   \tau   }  d\tau,\,f\in L_{2}(\Omega).
 $$
 Consider
  $$
  I_{1}(x)= \int\limits_{-\infty}^{x}f(\tau) e^{ \sqrt{2n} \tau }  d\tau,\;I_{2}(x)=\int\limits_{x}^{\infty}f(\tau) e^{ -\sqrt{2n}   \tau   }  d\tau.
  $$
 Observe that the functions $ f(x) e^{ \sqrt{2n} x },\, f(x) e^{ -\sqrt{2n}x} $ have the same Lebesgue points, then in accordance with the known fact, we have
$ I'_{1}(x)=f(x) e^{ \sqrt{2n} x },\;I'_{2}(x)=-f(x) e^{ -\sqrt{2n}x},$ where $x$ is a  Lebesgue point. Using this result, we get
$$
 \left(J_{n}f(x)\right)'  =-n\int\limits_{-\infty}^{x}f(\tau) e^{- \sqrt{2n} ( x-\tau) }  d\tau+n   \int\limits_{x}^{\infty}f(\tau) e^{- \sqrt{2n} ( \tau-x)  }  d\tau \;\mathrm{a.e.}
 $$
  Analogously, we have almost everywhere
$$
 \left(J_{n}f(x)\right)''  =n  \left\{  \sqrt{2n}\int\limits_{-\infty}^{x}f(\tau) e^{- \sqrt{2n} (x-\tau) }  d\tau+\sqrt{2n}\int\limits_{x}^{\infty}f(\tau) e^{- \sqrt{2n} (\tau-x)  }  d\tau   -2f(x)   \right\}=
 $$
$$
=2n(J_{n}-I)f(x)=-2AJ_{n}f(x),
$$
  taking into account the fact  $\mathrm{R}(J_{n})=\mathrm{R}(R_{A}(n))=\mathrm{D}(A),$  we obtain the desired result.

In accordance with the reasonings of  \cite[p.336]{firstab_lit:Yosida}, we have $C(\Omega)\subset \mathrm{D}(A).$ Denote by $A_{0}$ a restriction of $A$ on $C^{\infty}_{0}(\Omega).$ Using Lemma \ref{L1}, we conclude that $\tilde{A}_{0}=A,$ since there does not exist an accretive extension of an m-accretive operator.  Now, it is clear that
\begin{equation}\label{37.394}
\| A f\|_{L_{2}}\leq \|f\|_{H^{2,\,5}_{0} },\,f\in H^{2,\,5}_{0}( \Omega),
\end{equation}
whence $H^{2,\,5}_{0}( \Omega)\subset \mathrm{D}( A).$
Let us establish the representation $L= Z^{\alpha}_{G,F}(J).$  Since     the operator  $A$ is m-accretive, then using formula  \eqref{9},  we can define positive fractional powers $\alpha\in (0,1)$ of the operator $ A . $
       Applying  the   relations  obtained  above, we can calculate
\begin{equation}\label{37.014}
(\lambda I+ A )^{-1} A f(x)=-\frac{1}{2\sqrt{2\pi  }}\int\limits_{0}^{\infty}e^{-\lambda t}  t^{-1/2} dt
\int\limits_{-\infty}^{\infty} e^{-\frac{(x-\tau)^{2}}{2t}}f''(\tau)d\tau=
$$
$$
=-\frac{1}{2\sqrt{2\pi  }}\int\limits_{-\infty}^{\infty}f''(\tau)d\tau
\int\limits_{0}^{\infty}   e^{-\lambda t-\frac{(x-\tau)^{2}}{2t}}t^{-1/2}dt,\;f\in C_{0}^{\infty}(\Omega).
\end{equation}
Here,   substantiation of the interchange of the integration order  can be easily obtained due to the  properties of the function. We have for arbitrary chosen $x,\lambda$
$$
\int\limits_{-A}^{A}f''(\tau)d\tau
\int\limits_{0}^{\infty}   e^{-\lambda t-\frac{(x-\tau)^{2}}{2t}}t^{-1/2}dt=\int\limits_{-A-x}^{A-x}f''(x+s)ds
\int\limits_{0}^{1}   e^{-\lambda t- s^{2}/2t }t^{-1/2}dt+
$$
$$
+\int\limits_{-A-x}^{A-x}f''(x+s)ds
\int\limits_{1}^{\infty}   e^{-\lambda t- s^{2}/2t }t^{-1/2}dt.
$$
Observe that  the inner integrals converge uniformly with respect to $s,$ it is also clear that the function under the integrals is continuous regarding to $s,t,$ except of the set of points $(s;t_{0}),\,t_{0}=0.$ Hence applying the well-known theorem of calculus, we obtain \eqref{37.014}. Consider

$$
\int\limits_{-\infty}^{\infty}f''(x+s)ds
\int\limits_{0}^{\infty}   e^{-\lambda t- s^{2}/2t }t^{-1/2}dt =2\lambda^{-1/2}\int\limits_{-\infty}^{\infty}f''(x+s)ds
\int\limits_{0}^{\infty}   e^{-\sigma^{2}- c^{2}/\sigma^{2} } d\sigma=I,
$$
where $c^{2}=s^{2}\lambda/2.$ Using   formula \eqref{37.012}, we obtain
$$
I=\sqrt{\pi}\lambda^{-1/2} \int\limits_{-\infty}^{\infty}f''(x+s) e^{-\sqrt{2\lambda}|s|}ds=\sqrt{\pi}\lambda^{-1/2} \int\limits_{0}^{\infty}f''(x+s) e^{-\sqrt{2\lambda} s }ds+
  \sqrt{\pi}\lambda^{-1/2} \int\limits_{0}^{\infty}f''(x-s) e^{-\sqrt{2\lambda}  s }ds.
  $$
  Thus, combining formulas \eqref{9},\eqref{37.014},  we conclude that
$$
 A  ^{\alpha}f(x) = - \frac{2^{-3/2} }{\Gamma(1-\alpha)\Gamma(\alpha)} \int\limits_{0}^{\infty}\lambda^{\alpha -3/2}d\lambda\int\limits_{-\infty}^{\infty}f''(x+s) e^{-\sqrt{2\lambda}|s|}ds,\,f\in C_{0}^{\infty}(\Omega).
$$
We  easily prove that
\begin{equation}\label{37.023}
\int\limits_{\varepsilon}^{\infty}f''(x+s)ds\int\limits_{0}^{\infty}\lambda^{\alpha -3/2}e^{-\sqrt{2\lambda} s }d\lambda =\int\limits_{0}^{\infty}\lambda^{\alpha -3/2}d\lambda\int\limits_{\varepsilon}^{\infty}f''(x+s)e^{-\sqrt{2\lambda} s }ds,\,f\in C_{0}^{\infty}(\Omega).
\end{equation}
Let us show that
\begin{equation}\label{37.024}
\int\limits_{0}^{\infty}\lambda^{\alpha -3/2}d\lambda\int\limits_{\varepsilon}^{\infty}f''(x+s)e^{-\sqrt{2\lambda} s }ds\rightarrow \int\limits_{0}^{\infty}\lambda^{\alpha -3/2}d\lambda\int\limits_{0}^{\infty}f''(x+s)e^{-\sqrt{2\lambda} s }ds,\,\varepsilon \rightarrow 0,
\end{equation}
we have
$$
\left|\int\limits_{0}^{\infty}\lambda^{\alpha -3/2}d\lambda\int\limits_{0}^{\varepsilon}f''(x+s)e^{-\sqrt{2\lambda} s }ds\right|\leq\|f''\|_{L_{\infty}}\int\limits_{0}^{\infty}\lambda^{\alpha -3/2}d\lambda\int\limits_{0}^{\varepsilon} e^{-\sqrt{2\lambda} s }ds=
$$
$$
= \frac{1}{\sqrt{2}}\|f''\|_{L_{\infty}}\int\limits_{0}^{\infty}\lambda^{\alpha - 2} \left( 1-e^{- \sqrt{2\lambda} \varepsilon }\right)d\lambda =
$$
$$
=
 \varepsilon^{2(1-\alpha)}2^{3/2-\alpha}\|f''\|_{L_{\infty}}\int\limits_{0}^{\infty}t^{2\alpha - 3} \left( 1-e^{- t }\right)dt
 \rightarrow 0,\,\varepsilon\rightarrow0,
$$
from what follows the desired result.
Using simple calculations, we get
\begin{equation}\label{37.025}
\int\limits_{0}^{\varepsilon}f''(x+s) ds\int\limits_{0}^{\infty}e^{-\sqrt{2\lambda} s }\lambda^{\alpha -3/2}d\lambda=
 $$
 $$
 =2^{3/2-\alpha}\Gamma(2\alpha-1) \int\limits_{0}^{\varepsilon}f''(x+s)s^{1-2\alpha} ds \leq C\|f''\|_{L_{\infty} }\varepsilon^{2(1-\alpha)}\rightarrow 0,\,\varepsilon\rightarrow0.
\end{equation}
In accordance with \eqref{37.023}, we can write
$$
\int\limits_{0}^{\infty}f''(x+s)ds\int\limits_{0}^{\infty}\lambda^{\alpha -3/2}e^{-\sqrt{2\lambda} s }d\lambda =\int\limits_{0}^{\infty}\lambda^{\alpha -3/2}d\lambda\int\limits_{\varepsilon}^{\infty}f''(x+s)e^{-\sqrt{2\lambda} s }ds+\int\limits_{0}^{\varepsilon}f''(x+s) ds\int\limits_{0}^{\infty}e^{-\sqrt{2\lambda} s }\lambda^{\alpha -3/2}d\lambda.
$$
  Passing to the limit at the right-hand side, using \eqref{37.024},\eqref{37.025},  we obtain
$$
 \int\limits_{0}^{\infty}\lambda^{\alpha -3/2}d\lambda\int\limits_{0}^{\infty}f''(x+s)e^{-\sqrt{2\lambda} s }ds=\int\limits_{0}^{\infty}f''(x+s)ds\int\limits_{0}^{\infty}\lambda^{\alpha -3/2}e^{-\sqrt{2\lambda} s }d\lambda=
  $$
  $$
  =2^{3/2-\alpha}\Gamma(2\alpha-1) \int\limits_{0}^{\infty}f''(x+s)s^{1-2\alpha} ds  .
$$
Taking into account the analogous reasonings, we conclude that
$$
 A ^{\alpha}f(x)= - \frac{\Gamma(2\alpha-1)}{2^{  \alpha}\Gamma(\alpha)\Gamma(1-\alpha)}  \int\limits_{-\infty}^{\infty}f''(x+s)|s|^{1-2\alpha} ds=  K_{\alpha} I^{\alpha}f''(x),\;
 $$
 $$
 K_{\alpha}=-\frac{    \Gamma(2\alpha-1)\cos  \alpha \pi / 2 }{2^{  \alpha-1} \Gamma(1-\alpha)},\,f\in C_{0}^{\infty}(\Omega).
$$

Using   the  Hardy-Littlewood theorem with limiting exponent (see Theorem 5.3 \cite[p.103]{firstab_lit:samko1987}),  we get
\begin{equation}\label{37.028}
  \|A ^{\alpha}f\|_{L_{2}}\leq C\|I^{2(1-\alpha)}_{+}f''\|_{L_{2}}+C\|I^{2(1-\alpha)}_{-}f''\|_{L_{2}} \leq C \|f''\|_{L_{q}} ,\, f  \in C^{\infty}_{0}(\Omega),
\end{equation}
where $q=2/(5-4 \alpha ).$   Applying the H\"{o}lder inequality, we obtain
\begin{equation}\label{37.727}
\left(\int\limits_{-\infty}^{\infty}|f''(x)|^{q }(1+|x|)^{5q /2}(1+|x|)^{-5q /2}dx\right)^{1/q }\leq
$$
$$
\leq\left(\int\limits_{-\infty}^{\infty}|f''(x)|^{2}(1+|x|)^{5} dx\right)^{1/2}\left(\int\limits_{-\infty}^{\infty} (1+|x|)^{-5q \gamma/2}dx\right)^{1/q \gamma}\leq C\|f\|_{H_{0}^{2,5}},\,f \in C^{\infty}_{0}(\Omega),
\end{equation}
where $1<q<2,\,\gamma=2/(2 -q )>1.$
  Combining  \eqref{37.028},\eqref{37.727} and passing to the limit, we get
\begin{equation}\label{37.027}
  \|A ^{\alpha}f\|_{L_{2}}\leq C \|f\|_{H_{0}^{2,5}},\,f\in H_{0}^{2,5}(\Omega).
\end{equation}
Hence  $ H_{0}^{2,5}(\Omega)\subset \mathrm{D}(A ^{\alpha}).$
Using   the  Hardy-Littlewood theorem with limiting exponent, we obtain
 $$
 \|I^{\sigma}_{+}\rho I^{2(1-\alpha)} f''\|_{L_{2}}\leq C \| \rho I^{2(1-\alpha)} f''\|_{L_{q_{1}}}\leq C_{\rho}\|     f''\|_{L_{q_{2}}},\,f\in C_{0}^{\infty}( \Omega),\,C_{\rho}=C\|\rho\|_{L_{\infty}},
 $$
 where $q_{1}=2/(1+2\sigma),\,q_{2}=q_{1}/(1+2q_{1}[1-\alpha]).$   We can rewrite  $q_{2}=2/(1+2\sigma +4 [1-\alpha]),$ thus  $1<q_{2}<2.$
Applying  formula \eqref{37.727} and passing to the limit, we get
 \begin{equation}\label{39.917}
 \|I^{\sigma}_{+}\rho I^{2(1-\alpha)} f''\|_{L_{2}}\ \leq C_{\rho}\|     f \|_{H^{2,5}_{0}},\,f\in H_{0}^{2,5}(\Omega).
 \end{equation}
Note that
 \begin{equation}\label{37.127}
 \int\limits_{\Omega}  \mathcal{T }f\, \bar{g}  dx= \int\limits_{\Omega}    a   f''\,  \overline{g''}  dx,\; f,g\in C_{0}^{\infty}( \Omega).
 \end{equation}
Therefore  $  \mathcal{T} $ is accretive,   applying  Lemma \ref{L1}   we  deduce that $\tilde{\mathcal{T}}$ is m-accretive.
Using  relation \eqref{37.394},\eqref{37.127} we can easily obtain
$
\|\tilde{\mathcal{T }}f\|_{L_{2}}\geq  \gamma_{a}\|f\|_{H_{0}^{2,5}}\geq C \|Af\|_{L_{2}},\; f \in \mathrm{D}(\tilde{\mathcal{T }}),
$
whence $\mathrm{D}(\tilde{\mathcal{T }})\subset H_{0}^{2,5}( \Omega) \subset \mathrm{D}(A).$
    Using simple reasonings, we can extend  relation \eqref{37.127}  and rewrite it in the following form
$$
 \int\limits_{\Omega}  \tilde{\mathcal{T }}f\,\bar{g}  dx= \frac{1}{4}\int\limits_{\Omega}    a   Af\,   \overline{Ag}   dx,\; f\in \mathrm{D}(\tilde{\mathcal{T}}),\,g\in \mathrm{D}(A),
 $$
whence $\tilde{\mathcal{T}}\subset A^{\ast}G A,$ where $G:=a/4.$ Since the operator $\tilde{\mathcal{T}}$ is m-accretive, $A^{\ast}G A$ is accretive, then $\tilde{\mathcal{T}}= A^{\ast}G A.$
Hence, taking into account the inclusion $\mathrm{D}(\tilde{\mathcal{T }})\subset H_{0}^{2,5}( \Omega),$ relation \eqref{39.917}, we conclude that
$
L= A^{\ast}G A+F A^{\alpha},
$
where $ F:=\rho I.$

Let us prove that the operator $L$ satisfies     conditions H1--H2.     Choose the space  $L_{2}(\Omega)$ as a space $\mathfrak{H},$ the set  $C_{0}^{\infty}(\Omega)$ as a linear  manifold $\mathfrak{M},$ and the space  $H^{2,5}_{0}(\Omega)$ as a space $\mathfrak{H}_{+}.$    By virtue of     Theorem 1 \cite{firstab_lit:1Adams}, we have  $H^{2,5}_{0}(\Omega)\subset\subset L_{2}(\Omega).$
Thus, condition  H1  is satisfied.

Using   simple  reasonings (the proof is omitted), we come to the following inequality
 \begin{equation*}\label{27.1}
\left|\int\limits_{-\infty}^{\infty}  (\tilde{\mathcal{T}}+\delta I) f \cdot  \bar{g} \, dx\right|\leq C\|f\|_{H^{2,5}_{0}}\|g\|_{H^{2,5}_{0}},\; f,g\in C_{0}^{\infty}( \Omega).
\end{equation*}
Applying    the Cauchy Schwarz inequality, relation  \eqref{39.917},  we obtain
\begin{equation}\label{37.223}
 |\left(I^{\sigma}_{+}\rho I^{2(1-\alpha)} f'', g\right)_{ L _{2}}|  \leq C_{\rho} \|f\|_{H_{0}^{2,5}}\|g\|_{H_{0}^{2,5}},\,f ,\,g\in C^{\infty}_{0}(\Omega).
\end{equation}
On the other hand, using the conditions  imposed on the function $a(x),$  it is not hard to prove that
\begin{equation*}\label{27.2}
 \mathrm{Re}( [\tilde{\mathcal{T}}+\delta I]f,f)\geq \min\{\gamma_{a},\delta\}\|f\|^{2}_{H^{2,5}_{0}},\,f\in C_{0}^{\infty}( \Omega).
\end{equation*}
Using relation \eqref{37.223}, we can easily obtain
$$
\mathrm{Re}( I^{\sigma}_{+}\rho I^{2(1-\alpha)} f'',f)\geq -C_{\rho}\|f\|^{2}_{H^{2,5}_{0}},\,f\in C_{0}^{\infty}( \Omega).
$$
Combining the above estimates,  we conclude that   if the condition   $\min\{\gamma_{a},\delta\}>C_{\rho}$ holds,  then   $\mathrm{Re}( L f,f)\geq C \|f\|^{2}_{H^{2,5}_{0}},\,f\in C_{0}^{\infty}( \Omega).$ Thus, condition H2 is satisfied.
\end{proof}

\vspace{0.5cm}

\noindent{\bf Difference operator}\\

Consider a   space $L_{2}(\Omega),\,\Omega:=(-\infty,\infty),$   define a family of operators
$$
T_{t}f(x):=e^{-\lambda t}\sum\limits_{k=0}^{\infty}\frac{(\lambda t)^{k}}{k!}f(x-k\mu),\,f\in L_{2}(\Omega),\;\lambda,\mu>0,\; t\geq0,
$$
where convergence is understood in the sense of $L_{2}(\Omega)$ norm. It is not hard to prove that $T_{t}: L_{2}\rightarrow L_{2},$ for this purpose it is sufficient to note that
\begin{equation}\label{38.03}
\left\|\sum\limits_{k=n}^{n+p}\frac{(\lambda t)^{k}}{k!}f(\cdot-k\mu) \right\|_{L_{2}}\leq  \left\|  f \right\|_{L_{2}}\sum\limits_{k=n}^{n+p}\frac{(\lambda t)^{k}}{k!}.
\end{equation}
\begin{lem}
$T_{t}$ is a $C_{0}$ semigroup of contractions, the corresponding  infinitesimal generator and its adjoint operator are defined by the following expressions
$$
Af(x)=\lambda[f(x)-f(x-\mu)],\,A^{\ast}f(x)=\lambda[f(x)-f(x+\mu)],\,f\in L_{2}(\Omega).
$$
\end{lem}
\begin{proof}
Assume that $f\in L_{2}(\Omega).$  Analogously to  \eqref{38.03}, we easily  prove that $\|T_{t}f\|_{L_{2}}\leq \|f\|_{L_{2}}.$ Consider
$$
T_{s}T_{t}f(x)=e^{-\lambda s}\sum\limits_{n=0}^{\infty}\frac{(\lambda s)^{n}}{n!}\left[e^{-\lambda t}\sum\limits_{k=0}^{\infty}\frac{(\lambda t)^{k}}{k!}f(x-k\mu-n\mu)\right].
$$
Since we have
$$
 \left\|\sum\limits_{k=0}^{m}\frac{(\lambda t)^{k}}{k!}f(x-k\mu)\right\|_{L_{2}}\leq \|    f \|_{L_{2}}  \sum\limits_{k=0}^{m}\frac{(\lambda t)^{k}}{k!},
$$
then similarly  to the case corresponding to $C(\Omega)$ norm (the prove is based upon the properties of the absolutely convergent double series, see Example 3 \cite[p.327]{firstab_lit:Yosida} ), we conclude that
$$
T_{s}T_{t}f(x)=e^{-\lambda  s }\sum\limits_{n=0}^{\infty}\frac{(\lambda s)^{n}}{n!} \left[  e^{-\lambda  t } \sum\limits_{k=0}^{\infty}\frac{(\lambda t)^{k}}{k!} f(x-k\mu-n\mu )\right]=
$$
$$
=e^{-\lambda (s+t)}\sum\limits_{p=0}^{\infty}\frac{1}{p!} \left[p! \sum\limits_{n=0}^{p}\frac{(\lambda s)^{n}}{n!}\frac{(\lambda t)^{p-n}}{(p-n)!}f(x-p\mu )\right]=
$$
$$
 =e^{-\lambda (s+t)}\sum\limits_{p=0}^{\infty}\frac{1}{p!}  (\lambda s+\lambda t)^{p} f(x-p\mu ) =T_{s+t}f(x),
$$
where equality is understood in the sense of $L_{2}(\Omega)$  norm. Let us establish the strongly continuous property. For sufficiently small $t,$ we have
$$
\|T_{t}f-f\|_{L_{2}}\leq e^{-\lambda t}( e^{\lambda t}-1)\|f\|_{L_{2}}+e^{-\lambda t}\left\|\sum\limits_{k=1}^{\infty}\frac{(\lambda t)^{k}}{k!}f(\cdot\,-k\mu)\right\|_{L_{2}}\leq    t e^{-\lambda t}\|f\|_{L_{2}} \left\{C+\sum\limits_{k=0}^{\infty}\frac{(\lambda )^{k+1}t^{k}}{(k+1)!}\right\},
$$
from what  follows that
$$
\|T_{t}f-f\|_{L_{2}}\rightarrow 0,\,t\rightarrow 0.
$$
Taking into account the above facts, we conclude that $T_{t}$ is a $C_{0}$  semigroup of contractions. Let us show that
$$
Af(x)=\lambda[f(x)-f(x-\mu)],
$$
we have (the proof is omitted)
$$
\frac{(I-T_{t})f(x)}{t}=   \frac{1-e^{-\lambda t}}{t} f(x)- \lambda e^{-\lambda t}f(x-\mu)-te^{-\lambda t}\sum\limits_{k=2}^{\infty}\frac{ \lambda   ^{k}t^{k-2}}{k!}f(x-k\mu).
$$
Hence
$$
 \frac{(I-T_{t})f }{t} \stackrel{L_{2}}{ \longrightarrow} \lambda[f -f(\cdot\,-\mu)],\,t\downarrow0,
$$
thus, we have obtained the desired result.
 Using   change of    variables  in integral it is easy to show that
$$
\int\limits_{-\infty}^{\infty}Af(x)g(x)dx=\int\limits_{-\infty}^{\infty}f(x)\lambda[g(x)-g(x+\mu)]dx,\,f,g\in L_{2}(\Omega),
$$
hence  $A^{\ast}f(x)=\lambda[f(x)-f(x+\mu)],\,f\in L_{2}(\Omega).$  The proof is complete.
\end{proof}

It is remarkable that there are some difficulties to apply theorems  (A)-(C)  to a transform  $Z^{\alpha}_{aI,bI}(A),$ where $a,b$ are functions,   and the main of them can be said as follows "it is not clear how we should    build a space $\mathfrak{H}_{+}$". However we can consider a rather  abstract  perturbation of the above transform in order to reveal its spectral properties.

\begin{teo} Assume that  $Q$ is a   closed operator acting in $L_{2}(\Omega),\,Q^{-1}\in \mathcal{K}(L_{2}),$ the operator $N$ is strictly accretive, bounded, $\mathrm{R}(Q)\subset \mathrm{D}(N).$ Then
a perturbation
$$
L:= Z^{\alpha}_{aI,bI}(A)+ Q^{\ast}N Q ,\;a,b\in L_{\infty}(\Omega),\,\alpha\in(0,1)
$$
   satisfies conditions  H1--H2, if  $\gamma_{N}>\sigma\|Q^{-1}\|^{2},$
where we put $\mathfrak{M}:=\mathrm{D}_{0}(Q),$
$$
 \sigma= 4\lambda\|a\|_{L_{\infty}}+\|b\|_{L_{\infty}}\frac{\alpha\lambda^{\alpha}   }{\Gamma(1 -\alpha)}
 \sum\limits_{k=0}^{\infty}\frac{ \Gamma(k -\alpha)}{k! }.
$$
\end{teo}

\begin{proof}
Let us find a representation for fractional powers of the operator $A.$ Using the Balakrishnan formula (5) \cite[p.260]{firstab_lit:Yosida}, we get
\begin{equation}\label{27.17}
   A^{\alpha}f=\sum\limits_{k=0}^{\infty}C_{k}f(x-k\mu), \,f\in C^{\infty}_{0}(\Omega),
\end{equation}
$$
   \,C_{k}=-\frac{\alpha \lambda^{k} }{k!\Gamma(1-\alpha)}\int\limits_{0}^{\infty}e^{-\lambda t}t^{k-1-\alpha}dt=-\frac{\alpha\Gamma(k -\alpha)}{k!\Gamma(1 -\alpha)}\lambda^{\alpha},\,k=0,1,2,...,\,.
$$
Let us  extend    relation \eqref{27.17} to $L_{2}(\Omega).$    We have almost everywhere
$$
\sum\limits_{k=0}^{\infty}C_{k} g (x-k\mu)  -\sum\limits_{k=0}^{\infty}C_{k}  f (x-k\mu)=\sum\limits_{k=0}^{\infty}C_{k}[g(x-k\mu)-f (x-k\mu)],\,g\in C_{0}^{\infty}(\Omega),\,f\in L_{2}(\Omega),
$$
since the first sum is a partial sum for a fixed $x\in \mathbb{R}.$
  In accordance with formula (1.66) \cite[p.17]{firstab_lit:samko1987}, we have    $|C_{k}|\leq C \,k^{-1-\alpha},$  hence
$$
\left\|\sum\limits_{k=0}^{\infty}C_{k}[g(\cdot-k\mu)-f (\cdot-k\mu)]\right\|_{L_{2}}\leq\|g-f\|_{L_{2}}\sum\limits_{k=0}^{\infty}|C_{k}| .
$$
Thus, we obtain
$$
\forall f\in L_{2}(\Omega),\,\exists \{f_{n}\}\in C_{0}^{\infty}(\Omega):     \,f_{n}\stackrel{L_{2}}{ \longrightarrow} f,\;A^{\alpha}f_{n}\stackrel{L_{2}}{ \longrightarrow} \sum\limits_{k=0}^{\infty}C_{k}  f (\cdot-k\mu).
$$
Since $A^{\alpha}$ is closed, then
 \begin{equation}\label{27.171}
   A^{\alpha}f=\sum\limits_{k=0}^{\infty}C_{k}f(x-k\mu),\, f\in L_{2}(\Omega).
\end{equation}
Moreover, it is clear that $C^{\infty}_{0}(\Omega)$ is a core of $A^{\alpha}.$
 On the other hand, applying formula \eqref{9}, using the notation $\eta(x)=\lambda[f(x)-f(x-\mu)],$ we get
$$
A^{\alpha}f(x)=\frac{\sin\alpha \pi}{\pi}\int\limits_{0}^{\infty}\xi^{\alpha-1}(\xi I+A)^{-1}Af(x) d\xi=\frac{\sin\alpha \pi}{\pi}\int\limits_{0}^{\infty}\xi^{\alpha-1}d \xi\int\limits_{0}^{\infty}e^{-\xi t}T_{t}\eta(x)dt=
$$
$$
 =\frac{\sin\alpha \pi}{\pi}\sum\limits_{k=0}^{\infty}\frac{ \lambda   ^{k}}{k!}\eta(x-k\mu)\int\limits_{0}^{\infty}\xi^{\alpha-1} d \xi\int\limits_{0}^{\infty} e^{-t(\xi+\lambda)   }t^{k} dt=
$$
$$
 =\frac{\sin\alpha \pi}{\pi}\sum\limits_{k=0}^{\infty}\frac{ \lambda   ^{k}}{k!}\eta(x-k\mu)\int\limits_{0}^{\infty}\xi^{\alpha-1}(\xi+\lambda)^{-k-1}d \xi\int\limits_{0}^{\infty} e^{-t   }t^{k} dt,\,f\in C_{0}^{\infty}(\Omega),
$$
  we can rewrite the previous relation as follows
\begin{equation}\label{27.19}
A^{\alpha}f(x)=\sum\limits_{k=0}^{\infty} C'_{k}[f(x-k\mu)-f(x-(k+1)\mu)],\,f\in C_{0 }^{\infty}(\Omega),
\end{equation}
$$
\, C'_{k}=\frac{\lambda   ^{k+1}\sin\alpha \pi}{\pi}\int\limits_{0}^{\infty}\xi^{\alpha-1}(\xi+\lambda)^{-k-1}d \xi.
$$
Note that analogously to \eqref{27.171} we can extend formula \eqref{27.19} to $L_{2}(\Omega).$
Comparing  formulas \eqref{27.17},\eqref{27.19} we   can check the results
   calculating directly, we  get
$$
C'_{k+1}-C'_{k}
=-\frac{\lambda   ^{k+1}\sin\alpha \pi}{\pi} \int\limits_{0}^{\infty}\xi^{\alpha }(\xi+\lambda)^{-k-2}d \xi=
-\frac{\alpha\Gamma(k+1 -\alpha)}{(k+1)!\Gamma(1 -\alpha)}\lambda^{\alpha}=C_{k+1},
 \,C'_{0}=C_{0},\,k\in \mathbb{N}_{0}.
$$
Observe that by virtue of the made assumptions regarding   $Q,$ we have $\mathfrak{H}_{Q}\subset\subset L_{2}(\Omega).$ Choose the space  $L_{2}(\Omega)$ as a space $\mathfrak{H}$ and the space  $\mathfrak{H}_{Q} $ as a space $\mathfrak{H}_{+}.$  Let $ S:=Z^{\alpha}_{aI,bI}(A),\,T:=Q^{\ast}N Q.$  Applying the reasonings of Theorem  \ref{T1}, we conclude that there exists a set $\mathfrak{M}:=\mathrm{D}_{0}(Q),$ which is dense in $\mathfrak{H}_{Q},$ such that the operators $S,T$ are defined on its elements.
        Thus, we obtain the fulfilment of condition H1.
Since the operator $N$ is bounded, then
$
|(Tf,g)|_{L_{2}}\leq\|N\| \cdot \|f\|_{\mathfrak{H}_{Q}}\|g\|_{\mathfrak{H}_{Q}}.
$
Using formula \eqref{27.171}, we can easily obtain
$
|(Sf,g)|_{L_{2}}\leq \sigma \|f\|_{L_{2}}\|g\|_{L_{2}}\leq \sigma \|Q^{-1}\|^{2}\cdot\|f\|_{\mathfrak{H}_{Q}}\|g\|_{\mathfrak{H}_{Q}},\,\sigma=4\lambda\|a\|_{L_{\infty}}+ \|b\|_{L_{\infty}}\sum _{k=0}^{\infty} |C_{k}|.
$
Using the strictly accretive property of the operator $N$ we get
$
\mathrm{Re}(Tf,f)\geq \gamma_{N}\|f\|^{2}_{\mathfrak{H}_{Q}}.
$
On the other hand
$
\mathrm{Re}(Sf,f)\geq -\sigma \|Q^{-1}\|^{2}\cdot\|f\|^{2}_{\mathfrak{H}_{Q}},
$
hence condition H2 is satisfied. The proof is complete.

\end{proof}

\section{Conclusions}
In this paper, we   studied  a true   mathematical nature of a differential operator   with a fractional derivative in final terms. We constructed a   model  in terms of the infinitesimal generator of a corresponding semigroup and    successfully applied   spectral theorems. Further, we generalized the obtained results to some class of transforms of  m-accretive operators, what can be treated as  an introduction to the fractional calculus of m-accretive operators. As a concrete theoretical achievement  of the offered  approach, we have  the following results: an  asymptotic equivalence between   the
real component of a  resolvent and the resolvent of the real component was established  for the  class; a  classification,   in accordance with     resolvent  belonging      to   the  Schatten-von Neumann  class, was obtained;
  a  sufficient condition of completeness of the root vectors system were formulated; an asymptotic formula for the eigenvalues was obtained. As an application,
 there were considered cases corresponding to a finite and infinite measure as well as various notions of fractional derivative under the semigroup  theory  point of view, such operators as a  Kipriyanov operator, Riesz potential,  difference operator were involved. The eigenvalue problem for
  a differential operator   with a composition of fractional integro-differential operators  in final terms was solved.

 In addition, note that  minor  results  are also worth noticing such as  a generalization of the well-known von Neumann theorem   (see  the proof of    Theorem  \ref{T1}).    In   section 3,   it  might have been    possible to   consider an unbounded domain  $\Omega$ with some restriction imposed upon    a solid angle  containing  $\Omega,$ due to this natural  way we come to a   generalization of the Kipriyanov operator.   We should add  that  various   conditions, that may be imposed on the operator $F,$ are worth studying separately  since there is a number   of applications in  the theory of fractional differential equations.

\vspace{0.5 cm}
\noindent{\bf Acknowledgments}\\

 The author warmly thanks  academician  Andrey A. Shkalikov for  valuable comments and remarks.

 Gratitude is expressed to  professor Virginia  Kiryakova for a kindly given   invaluable bibliographic survey.

\end{document}